\documentclass[11pt]{article}

\usepackage{hyperref}
\usepackage{blkarray}

\usepackage{amsmath}
\usepackage{amssymb}
\usepackage{amsfonts}
\usepackage{amsthm}
\usepackage{enumerate}
\usepackage{multicol}
\usepackage{colortbl}

\usepackage{pgf}
\usepackage{tikz}
\usepackage{xcolor}
\usetikzlibrary{arrows}
\usetikzlibrary{patterns}
\usetikzlibrary{petri}

\usepackage[english]{babel}
\usepackage{tabularx}
\usepackage{subfig}

\newcommand{\rank}{{\rm rank}}
\newcommand{\bt}{t}
\newcommand{\p}{p}
\newcommand{\bq}{{q}}
 
\newcommand{\bu}{{u}}
\newcommand{\bm}{m}
\newcommand{\bn}{{n}}

\newcommand{\A}{\mathcal A}
\newcommand{\C}{\mathcal C}
\newcommand{\M}{\mathcal M}
\newcommand{\gs}{\M_{\C}} 

\newcommand{\pogp}{\langle G', \bm' \rangle}

\newcommand{\pogn}{\langle G, \bn \rangle}

\newcommand{\tg}{\langle G, \bm \rangle}

\newcommand{\tgn}{\langle G, \bn \rangle}

\newcommand{\tgT}{\langle G, \bm_T \rangle}

\newcommand{\Tor}{\mathcal T_0} 
\newcommand{\Unit}{\mathcal U}	
\newcommand{\T}{\mathcal T} 

\newcommand{\wt}{\widetilde}
\newcommand{\pfwflr}{(\langle \wt G, L \rangle , R\wt \p)} 
\newcommand{\pfwf}{(\langle \wt G, L_0 \rangle , \wt \p)} 

\newcommand{\pofw}{(\langle G, \bm \rangle,\p)} 
\newcommand{\pofwT}{(\langle G, \bm_T \rangle,\p)} 
\newcommand{\pofwA}{(\langle G, \bm \rangle,A(\p))} 

\newcommand{\pog}{\langle G, \bm \rangle} 

\newcommand{\BS}{(\widetilde{G}, \Gamma, \widetilde{\p}, \pi)} 
\newcommand{\BSg}{(\widetilde{G}, \Gamma)} 

\newcommand{\dpfwfl}{( \langle G^{\bm}, L \rangle, \p^{\bm})} 
\newcommand{\dpfwo}{( \langle G^{\bm}, L_0 \rangle, \p^{\bm})} 
\newcommand{\dpfww}{( \langle G^{\bm}, \wt L \rangle, \p^{\bm})} 

\newcommand{\R}{\mathbf R} 

\definecolor{desk}{rgb}{.345, .306, .216}
\definecolor{vancouver}{rgb}{.412, .412,.412}
\definecolor{beetle}{rgb}{.180, .161, .102}
\definecolor{bluey}{rgb}{.235, .380, .415}
\definecolor{melon}{rgb}{1, .259, .259}
\definecolor{vneck}{rgb}{.596, .282, .376}
\definecolor{pink}{rgb}{.918, .122, .545}
\definecolor{mango}{rgb}{1, .8, .267}
\definecolor{lips}{rgb}{.541, .074, .239}
\definecolor{sage}{rgb}{.522, .604, .247}
\definecolor{moss}{rgb}{.184, .224, .129}
\definecolor{cumin}{rgb}{.6, .580, 0}
\definecolor{lichen}{rgb}{.745, .998, .729}
\definecolor{rain}{rgb}{.780, .812, .706 }
\definecolor{cloud}{rgb}{.961, .976, .870}
\definecolor{couch}{rgb}{.8, 1, .2}
\definecolor{cement}{rgb}{.678, .682, .549}
\definecolor{sky}{rgb}{.278, .514, 1}

\newtheorem{thm}{Theorem}[section]
\newtheorem{cor}[thm]{Corollary}
\newtheorem{lem}[thm]{Lemma}
\newtheorem{prop}[thm]{Proposition}
\newtheorem{conj}[thm]{Conjecture}

\theoremstyle{remark}
\newtheorem{rem}[thm]{Remark}
\theoremstyle{remark}
\newtheorem{ex}[thm]{Example}
\theoremstyle{remark}

\newtheorem{question}[thm]{Question}

\begin{document}

\title{The rigidity of periodic frameworks as graphs on a fixed torus}
\date{September 19, 2011}
\author{
{Elissa  Ross
\thanks{Fields Institute for Research in Mathematical Sciences, Toronto, Canada.}}
}
\maketitle

\begin{abstract} 
We define periodic frameworks as graphs on the torus, using the language of gain graphs. We present some fundamental definitions and results about the infinitesimal rigidity of graphs on a torus of fixed size and shape, and find necessary conditions for the generic rigidity of periodic frameworks on a $d$-dimensional fixed torus.
\\	


\end{abstract}


\section{Introduction}
Like many problems in the field of discrete geometry, the question of the {\it rigidity of a framework} admits a simple formulation. Given a set of physically rigid bars which are linked together by flexible joints, when is it possible to continuously deform the resulting framework into a non-congruent structure, without destroying the connectivity or the bars themselves? In other words, when is such a framework {\it flexible},  and therefore not {\it rigid}? 

The study of rigidity has a rich history of questions generated by applications in structural engineering, mechanical engineering (in the study of linkages),  chemistry, biology, materials science  and computing, which then inspire and motivate a body of mathematical research. The study of {\it periodic rigidity} can be seen as exactly such a case, with one of the main inspirations coming from the study of zeolites. Zeolites are a type of mineral with a crystalline structure characterized by a repetitive (periodic) porous pattern and a high internal surface area \cite{flexibilityWindow}. Since the activity of these materials in appliations appears to depend in part on their flexibility, it is desirable to have methods that would predict the rigidity or flexibility of these hypothetical minerals prior to laboratory synthesis.

Infinite periodic frameworks in $3$-space are used to model the molecular structure of zeolites, and as a result, there has been a recent surge of interest in the rigidity of these periodic structures. Examples of such work include Fowler and Guest   \cite{triangulatedToroidalFrameworks} and Guest and Hutchinson \cite{DeterminancyRepetitive}, both of which address two and three dimensional frameworks (with a view toward materials). \index{zeolites} 
Even more recently, work by Owen and Power \cite{InfiniteBarJointFrameworksCrystals}, Power \cite{PowerAffine}, Borcea and Streinu  \cite{BorceaStreinuII, periodicFrameworksAndFlexibility} and Malestein and Theran \cite{Theran} has  formalized the mathematics involved in  a general ($d$-dimensional) study of infinite periodic frameworks and provided substantial initial results. 

In this paper we describe one structure and corresponding vocabulary for this investigation. We outline results from a natural ``base case" for the study of general infinite periodic frameworks, namely frameworks on a torus of fixed size and shape. While at first the question of rigidity on a ``fixed torus" may seem contrived, several materials scientists have confirmed that there may be some resonance with experiments on molecular compounds in which the time scales of lattice movement are several orders of magnitude slower than the molecular deformations within the lattice \cite{ThorpePrivate}. When we allow the lattice (torus) to deform, the velocities of the vertices that are ``far away from the centre" will become arbitrarily large. 

The results of this paper will lay the groundwork for a subsequent paper \cite{ThesisPaper2} which will provide sufficient conditions for the generic rigidity of a 2-dimensional framework on a fixed torus. In addition, the structures and vocabulary contained here are employed in recent joint work with Bernd Schulze and Walter Whiteley \cite{SymmetryPeriodic} concerning periodic framework with additional symmetry. 

The central idea that underlies our research is to exploit the periodicity of the infinite graph to reduce the problem to a finite graph that captures the periodic structure. We accomplish this by considering quotient graphs on tori. For example, to study two-dimensional infinite periodic frameworks, we view the two-dimensional torus as a fundamental region for a tiling of the plane, and consider graphs realized on the torus as models of infinite periodic frameworks in the plane. Any motion of the elements of the framework on the torus can be viewed as a {\it periodic} motion of the plane graph. We can similarly consider graphs on the $d$-torus (equivalently the $d$-dimensional hypercube with pairs of opposite faces identified) and use this as a model of a $d$-dimensional periodic framework. {\it Gain graphs} \cite{TopologicalGraphTheory, gainGraphBibliography} provide a useful language for the description of these graphs, and the tools of topological graph theory will be used to show that graphs in the same homotopy class share the same rigidity properties (the $T$-gain procedure). 

There are three qualities of infinite periodic frameworks that are of interest to the study of their rigidity: 
\begin{enumerate}[(i)]
	\item the combinatorial properties of the graph,
	\item the geometric position of vertices of the graph on the torus, and in its cover in $d$-space,
	\item the topological structure (up to homotopy) of the graph on the torus. 
\end{enumerate}
The usual study of rigidity of finite frameworks (as described in \cite{CombRigidity,CountingFrameworks,SomeMatroids} for example) is an investigation of (i) and (ii), but the consideration of (iii) is unique to the study of periodic frameworks. It should be noted that the approach of Borcea and Streinu \cite{BorceaStreinuII,periodicFrameworksAndFlexibility} to the study of periodic frameworks does not explicitly include the consideration of topological structures as in (iii). Like the approach outlined in this paper, Malestein and Theran \cite{Theran} also choose to include these topological properties. Their  work is concerned with frameworks on a flexible torus, and is restricted to two-dimensional frameworks. In this paper we describe necessary conditions for the rigidity of periodic frameworks on the fixed torus in $d$-dimensions. Of course necessary conditions for rigidity can always be viewed as sufficient conditions for flexibility, something that we exploit in \cite{SymmetryPeriodic}.

We remark finally that this work is concerned with {\it forced periodicity}. That is, we are interested in motions of a periodic structure that preserve the periodicity of the structure. An infinite periodic framework may have motions that break the periodic symmetry of the framework, but we will not address these motions here. The consideration of periodicity-breaking motions would be the study of {\it incidental periodicity}, frameworks which happen to be periodic, but do not necessarily preserve their periodicity through some motion of their joints. That is a distinct, though important, topic.

\subsection{Outline of paper}

In Section~\ref{sec:background} we present some necessary background from topological graph theory, specifically introducing gain graphs. We also introduce the fixed torus. In Section~\ref{sec:periodicDefs} we define periodic frameworks as graphs on the torus, and show that this definition is sensible. Section~\ref{sec:rigidity} contains numerous periodic-adapted versions of standard definitions and results of rigidity theory, including a definition of the periodic rigidity matrix. We provide some context for the current approach in light of recent work in the area by other authors. We also prove that all graphs in the same homotopy class on the fixed torus have the same generic rigidity properties (Corollary \ref{cor:TgainsPreserveGeneric}). In Section~\ref{sec:necessary} we prove necessary conditions on periodic orbit frameworks for generic rigidity on the fixed $d$-dimensional torus (Theorem \ref{thm:dNecessary}). We conclude in Section~\ref{sec:furtherWork} with some areas for further work, and a preview of the sequel paper \cite{ThesisPaper2}. 

\section{Background}
\label{sec:background}
\subsection{Gain graphs}
\label{sec:gainGraphs}
{\it Gain graphs} are a concise way of describing infinite periodic graphs.  We also view a gain graph as a set of instructions for how to realize a graph on the torus, although this need not be a $2$-cell embedding (an embedding without crossings), which distinguishes this treatment from other discussions of gain graph realizations \cite{TopologicalGraphTheory}. Note that in some literature, namely Gross and Tucker's book \cite{TopologicalGraphTheory}, these graphs are called {\it voltage graphs}. \index{voltage graph|see{gain graph}} Our discussion here is based on that presentation, but we use the word `gain' to avoid the extra connotations given by the term `voltage,' and to connect to the larger body of literature on the topic of gain graphs \cite{gainGraphBibliography}. 


Let $G = (V, E)$ be a connected multigraph possibly having loops and multiple edges with vertices $V = \{v_0, v_1, \dots, v_{n}$\}, $|V| = n < \infty$. Let the edges of $G$ be assigned both plus and minus directions. Let $\bm$ be a set function from the plus-directed edges into $\mathbb Z^d$. The pair $\pog$ is called a {\it gain graph}.  $G$ is called the {\it base graph} \index{gain graph!base graph}of $\pog$, $\bm: E \rightarrow \mathbb Z^d$ is called the {\it gain assignment}. In general, gain graphs have edges which are labeled by elements of a group $\A$ (the {\it gain group}), but for reasons that will soon become clear, we will use $\A = \mathbb Z^d$ throughout this paper. 

The vertices of $\pog$ are the same as the vertices of $G$: $V\pog = \langle V, \bm \rangle = V$. The edges of $\pog$ are denoted $\langle E, \bm \rangle$ or $E\pog$. An edge $e$ in $\langle E, \bm \rangle$  is denoted
\begin{equation}
e = \{v_i, v_j; \bm_e\}, \textrm{\ or \ } \{i, j; m_e\},
\label{eqn:edge1}
\end{equation}
where $\{v_i, v_j\} \in E$. \index{gain graph!edges} This represents the directed edge from vertex $v_i$ to vertex $v_j$, which is labeled with the gain $\bm_e$. This edge may equivalently be written in the reverse order, by using the group inverse $m_e^{-1} = -m_e$ of the gain assignment on $e$:
\begin{equation}
e = \{v_j, v_i; \bm_e^{ -1}\}.
\label{eqn:edge2}
\end{equation}

A {\it subgraph} of $\pog$ is a gain graph $\pogp$ where $G' \subset G$ is a subgraph of $G$, and $m'$ is the restriction of $m$ to the edges of $G'$. \index{gain graph!subgraph} 

A {\it path} of $\pog$ is defined to be a path of the base graph $G$. \index{gain graph!path}  We record a path of $\pog$ by
\[P = e_1^{\alpha_1} e_2^{\alpha_2} \cdots e_k^{\alpha_k}, \]
where $e_i \in E\pog$, and $\alpha_i$ is either $+1$ or $-1$ depending on the orientation of the edge in the path. This allows us to define the {\it net gain on the path} to be the sum of the elements on the edges of the path, with the appropriate multiplier ($+1$ or $-1$) according to the orientation of the edges in the path:
\[\sum_{i=1}^k \alpha_i \bm(e_i).\]
We similarly define a {\it cycle} of $\pog$ to be a cycle of the base graph $G$, and the {\it net gain on the cycle} is defined as for paths. 

For example, consider the graph in Figure \ref{fig:gainCycle}. Suppose the edge $e_i$ has gain $m_i$, as labeled. Then the cycle in the graph shown in Figure \ref{fig:gainCycle} given by
\[e_1^{+1} e_2^{+1} e_4^{-1} e_5^{+1} = \{1, 2; m_1\} \{2, 3; m_2\} \{3, 4; -m_4\} \{4, 1; m_5\}\]
has net gain $m_1 + m_2 - m_4 + m_5$.

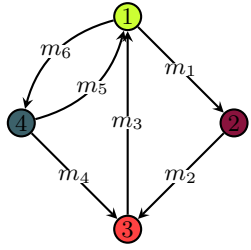
\begin{figure}\begin{center}
\begin{tikzpicture}[auto, node distance=2cm, thick]
\tikzstyle{vertex1}=[circle, draw, fill=couch, inner sep=1pt, minimum width=3pt, font=\footnotesize];
\tikzstyle{vertex2}=[circle, draw, fill=lips, inner sep=1pt, minimum width=3pt, font=\footnotesize];
\tikzstyle{vertex3}=[circle, draw, fill=melon, inner sep=1pt, minimum width=3pt, font=\footnotesize];
\tikzstyle{vertex4}=[circle, draw, fill=bluey, inner sep=1pt, minimum width=3pt, font=\footnotesize];
\tikzstyle{gain} = [fill=white, inner sep = 0pt,  font=\footnotesize, anchor=center];

\node[vertex1] (1) {$1$};
\node[vertex2] (2) [below right of=1] {$2$};
\node[vertex3] (3) [below left of=2] {$3$};
\node[vertex4] (4) [below left of=1] {$4$};


\pgfsetarrowsend{stealth}[ shorten >=1cm]
\path
(3) edge node[gain] {$m_3$} (1)
(1) edge [bend right] node[gain] {$m_6$} (4)
(1) edge node[gain] {$m_1$} (2)
(2) edge node[gain] {$m_2$} (3)
(4) edge [bend right] node[gain] {$m_5$}  (1)
(4) edge node[gain] {$m_4$} (3);
\pgfsetarrowsend{}

\end{tikzpicture}\end{center}
\caption{A gain graph $\pog$, $m: E \rightarrow \mathbb Z^d$. \label{fig:gainCycle}}
\end{figure}

The {\it edge space} \index{edge space} $\mathcal E(G)$  of a graph $G=(V, G)$ is the set of functions $E \rightarrow \mathbb F_2 = \{0,1\}$. The elements of $\mathcal E(G)$ are naturally associated with the subsets of $E$, however the edge set thus defined has the structure of a vector space. The elements are the subsets of $E$, vector addition is the same as symmetric difference, $\emptyset \subseteq E$ is the zero element, and $F = -F$ for all $F \in \mathcal E(G)$. See \cite{Diestel} for further details. 

The {\it cycle space} \index{cycle space} $\C = \C(G)$ of $G$ is the subspace of $\mathcal E(G)$ spanned by the (edge sets of the) cycles of $G$. Suppose $\pog$ is a gain graph where $\C(G)$ is the cycle space of the (undirected) graph $G$. The {\it gain space} $\gs(G)$ is the vector space (over $\mathbb Z$) spanned by the net gains on the cycles of $\C(G)$. \index{gain graph!gain space} \index{gain space} 

\begin{rem}
In contrast to cycles in directed graphs, we permit re-direction of the edges of a gain graph provided that they are accompanied by a relabelling of the gains on the edges as well (by the equivalence of (\ref{eqn:edge1}) and (\ref{eqn:edge2})). In this way, we should think of cycles in the gain graph as corresponding one-to-one with cycles in the base graph. That is, the gain graph $\pog$ has the same cycle space as the base graph $G$, while a directed graph does not share the same cycle space as its underlying graph. 
\label{rem:digraph} \qed
\end{rem}

\subsection{Derived graphs corresponding to gain graphs}

The key feature of gain graphs is that from a gain graph $\pog$ we may define a related graph called the {\it derived graph} \index{derived graph} which we denote $G^{\bm}$.  The derived graph $G^m$ has vertex set $V^m$ and $E^m$ where $V^{\bm}$ is the Cartesian product $V \times \mathbb Z^d$, and $E^{\bm}$ is the Cartesian product $E \times \mathbb Z^d$. Vertices of $V^{\bm}$ have the form  $(v_i, a)$, where $v_i \in V$, and $a \in \mathbb Z^d$. Edges of $E^{\bm}$ are denoted similarly. \index{derived graph!edges} If $e$ is the directed edge connecting vertex $v_i$ to $v_j$ in $\pog$, and $b$ is the gain assigned to the edge $e$, then the directed edge $\{e ;a \}$ of $G^{\bm}$ connects vertex $(v_i,a)$ to $(v_j, a+b)$. In this way, the derived graph is a (directed) graph whose automorphism group contains $\mathbb Z^d$. 

If $v$ is a vertex in the gain graph, then the set of vertices $\{(v,a): a \in \mathbb Z^d\}$ in the vertices $V^{\bm}$ of the derived graph is called the {\it fiber} \index{derived graph!fiber} over $v$. Similarly, the set of edges $\{(e,a): a \in \mathbb Z^d\}$ is the {\it fiber} over the edge $e \in E$. There is a {\it natural projection} from the derived graph to the base graph which is the graph map $\phi: G^{\bm} \rightarrow G$ that maps every vertex (resp. edge) in the fiber over $v$ (resp. $e$) to the vertex $v$ (resp. $e$) for all $v \in V$ (resp. $e \in E$).  Since $\mathbb Z^d$ is an infinite group, this representation allows us to view gain graphs as a `recipe' for an infinite periodic graph. 

\begin{ex}
Let $\pog$ be the gain graph pictured in Figure \ref{fig:finiteGraph}, with gain group $\mathbb Z^2$. The unlabeled, undirected edges have gain $(0,0)$. Part of the corresponding derived graph $G^m$ is pictured in \ref{fig:periodicFramework}. $G^m$ has a countably infinite number of vertices and edges. 
\qed
\end{ex}

\begin{verse}
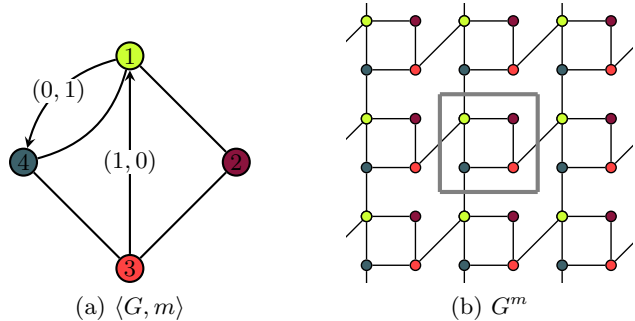
\begin{figure}[h!]
\begin{center}
\subfloat[$\pog$]{\label{fig:finiteGraph}\begin{tikzpicture}[auto, node distance=2cm, thick]
\tikzstyle{vertex1}=[circle, draw, fill=couch, inner sep=1pt, minimum width=3pt, font=\footnotesize];
\tikzstyle{vertex2}=[circle, draw, fill=lips, inner sep=1pt, minimum width=3pt, font=\footnotesize];
\tikzstyle{vertex3}=[circle, draw, fill=melon, inner sep=1pt, minimum width=3pt, font=\footnotesize];
\tikzstyle{vertex4}=[circle, draw, fill=bluey, inner sep=1pt, minimum width=3pt, font=\footnotesize];
\tikzstyle{gain} = [fill=white, inner sep = 0pt,  font=\footnotesize, anchor=center];

\node[vertex1] (1) {$1$};
\node[vertex2] (2) [below right of=1] {$2$};
\node[vertex3] (3) [below left of=2] {$3$};
\node[vertex4] (4) [below left of=1] {$4$};


\pgfsetarrowsend{stealth}[ shorten >=1cm]
\path
(3) edge node[gain] {$(1,0)$} (1)
(1) edge [bend right] node[gain] {$(0,1)$} (4);
\pgfsetarrowsend{}

\path (1) edge  (2)
(2) edge  (3)
(4) edge [bend right]   (1)
(4) edge  (3);

\end{tikzpicture}}\hspace{.5in}%
\subfloat[$G^m$]{\label{fig:periodicFramework}\includegraphics[width=1.5in]{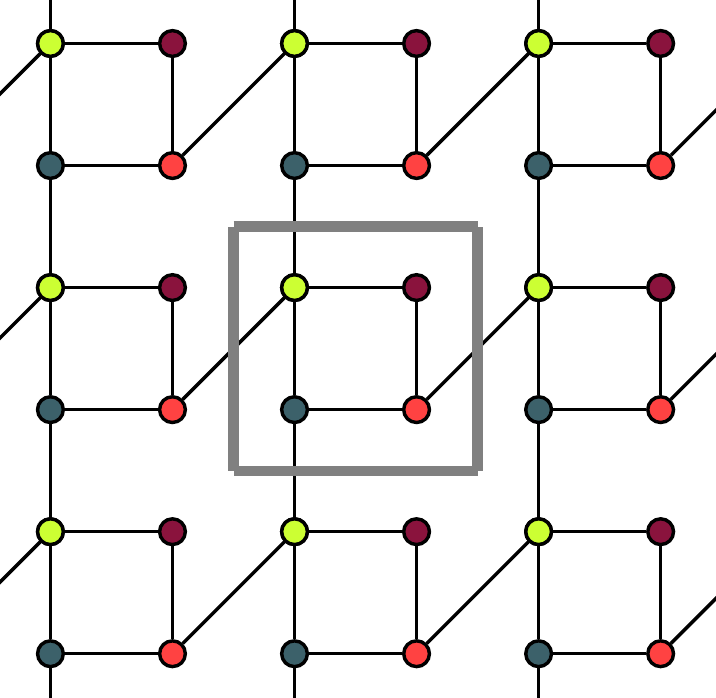}}

\caption{A gain graph $\pog$, where $m:E \rightarrow \mathbb Z^2$, and its derived graph $G^m$.  We use graphs with vertex labels as in (a) to depict gain graphs, and graphs without such vertex labels will record derived graphs, or graphs that are realized in $\mathbb R^d$. \label{fig:gainGraph}}
\end{center}
\end{figure}\end{verse}

\subsection{Local gain groups and the $T$-gain procedure}
\label{sec:Tgains} 
Let $u$ be a vertex of the gain graph $\pog$, and let $W$ and $W'$ be distinct closed walks that begin and end at $u$. The walk $WW'$ is also a $u$-based closed walk. The set of all such walks forms a semigroup, with the product operation so defined. It was observed by Alpert and Gross that the set of net gains occurring on $u$-based closed walks forms a subgroup of the gain group \cite{TopologicalGraphTheory}. We call this group the {\it local gain group at $u$}. \index{local gain group} \index{gain graph!local gain group} For a connected graph, it is clear that there is a unique local gain group that is independent of the choice of base vertex $u$. Furthermore, Gross and Tucker \cite{TopologicalGraphTheory} observe that we can extend the idea of local gain group to a notion of the fundamental group of a graph. They note that the standard topological theorems relating fundamental groups and covering spaces may be obtained for graphs. Furthermore, this justifies the use of the term ``graph homotopic" to describe the $T$-gain procedure, or any other transformation which preserves the cycles of a gain graph $\pog$ and their net gains. We now describe a procedure to isolate the local gain group of a gain graph. 



If our graph is a bouquet of loops, then the local gain group is simply the group generated by the gains of the loops. If our graph is not, however, a bouquet of loops, how do we find the local gain group? We have an algorithm called the {\it T-gain procedure} \index{T-gain procedure} that will effectively transform our graph into a bouquet of loops. It appears in \cite{TopologicalGraphTheory} and we outline it here. See Figure \ref{fig:Tvoltage} for a worked example.

\begin{verse}
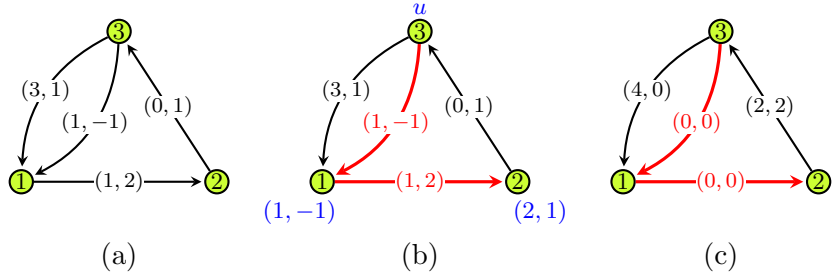
\begin{figure}[h!]
\begin{center}
\begin{tikzpicture}[->,>=stealth,shorten >=1pt,auto,node distance=2.8cm,thick, font=\footnotesize] 
\tikzstyle{vertex1}=[circle, draw, fill=couch, inner sep=.5pt, minimum width=3.5pt, font=\footnotesize]; 
\tikzstyle{vertex2}=[circle, draw, fill=melon, inner sep=.5pt, minimum width=3.5pt, font=\footnotesize]; 
\tikzstyle{voltage} = [fill=white, inner sep = 0pt,  font=\scriptsize, anchor=center];

	\node[vertex1] (1) at (-1.3,0)  {$1$};
	\node[vertex1] (2) at (1.3,0) {$2$};
	\node[vertex1] (3) at (0, 2) {$3$};
		\draw[thick] (1) -- node[voltage] {$(1,2)$} (2);
		\draw[thick] (2) -- node[voltage] {$(0,1)$} (3);
	\draw[thick] (3) edge  [bend right]  node[voltage] {$(3,1)$} (1);
	\draw[thick] (3) edge  [bend left] node[voltage] {$(1,-1)$} (1);
	
	\node[font=\normalsize] at (0, -1) {(a)};
	
	\pgftransformxshift{4cm};
	
		\node[vertex1] (1) at (-1.3,0)  {$1$};
	\node[vertex1] (2) at (1.3,0) {$2$};
	\node[vertex1] (3) at (0, 2) {$3$};
		\draw[very thick, red] (1) -- node[voltage] {$(1,2)$} (2);
		\draw[thick] (2) -- node[voltage] {$(0,1)$} (3);
	\draw[thick] (3) edge  [bend right]  node[voltage] {$(3,1)$} (1);
	\draw[very thick, red] (3) edge  [bend left] node[voltage] {$(1,-1)$} (1);
	
	\node[blue] at (0, 2.3) {$u$};
	\node[blue] at (-1.6, -.4) {$(1, -1)$};
	\node[blue] at (1.6, -.4) {$(2, 1)$};
	
	\node[font=\normalsize] at (0, -1) {(b)};
	
	\pgftransformxshift{4cm};
	
	\node[vertex1] (1) at (-1.3,0)  {$1$};
	\node[vertex1] (2) at (1.3,0) {$2$};
	\node[vertex1] (3) at (0, 2) {$3$};
		\draw[very thick, red] (1) -- node[voltage] {$(0,0)$} (2);
		\draw[thick] (2) -- node[voltage] {$(2,2)$} (3);
	\draw[thick] (3) edge  [bend right]  node[voltage] {$(4,0)$} (1);
	\draw[very thick, red] (3) edge  [bend left] node[voltage] {$(0,0)$} (1);
	
	\node[font=\normalsize] at (0, -1) {(c)};

\end{tikzpicture}
\caption{A gain graph $\pog$ in (a), with identified tree $T$ (in red), root $u$, and $T$-potentials in (b). The resulting $T$-gain graph $\langle G, \bm_T \rangle$ is shown in (c). The local gain graph is now seen to be generated by the elements $(4,0)$ and $(2,2)$, hence the local gain group is $2\mathbb Z \times 2\mathbb Z$. \label{fig:Tvoltage}}
\end{center}
\end{figure}\end{verse}


\noindent {\bf $T$-gain Procedure}
\begin{enumerate}
	\item Select an arbitrary spanning tree $T$ of $G$, and choose a vertex $u$ to be the root vertex (of the local gain group). Such a spanning tree is known to exist, as we assumed $G$ was connected.
	\item For every vertex $v$ in $G$, there is a unique path in the tree $T$ from the root $u$ to $v$. Denote the net gain along that path by $\bm(v, T)$, and we call this the {\it $T$-potential} of $v$. \index{T-gain procedure!T-potential} Compute the $T$-potential of every vertex $v$ of $G$. 
	\item Let $e$ be a plus-directed edge of $G$ with initial vertex $v$ and terminal vertex $w$. Define the {\it $T$-gain} of $e$, $\bm_T(e)$ to be 
	$$\bm_T(e) = \bm(v, T) +  \bm(e) - \bm(w, T).$$ 
Compute the $T$-gain of every edge in $G$. Note that the $T$-gain of every edge of the spanning tree will be zero. 
	\item Contract the graph along the spanning tree to obtain $|E| - (|V|-1)$ loops at the root vertex $u$ (there are $|V|-1$ edges as part of the spanning tree). The gains on these loops will generate the local gain group. In other words, the gains on all of the edges of the graph that are not contained in $T$ will generate the local gain group. 
\end{enumerate}

Since the net gain on any $u$-based closed walk is the same with respect to the $T$-gains as with respect to $\bm$, we have the following theorem: 

\begin{thm}[\cite{TopologicalGraphTheory}]
Let $\pog$ be a gain graph, and let $u$ be any vertex of $G$. Then the local gain group at $u$ with respect to the $T$-gains, for any choice of spanning tree $T$, is identical to the local group of $u$ with respect to $\bm$.  
\end{thm}

In other words, the $T$-gain procedure supplies us with the net gains on a fundamental system of cycles. It should be noted that different choices of $T$ will correspond to different fundamental systems of cycles, but all will generate the same cycle space and gain space. 

What is important for the study of rigidity is that the gain graph with $T$-gains generates the same derived graph as the gain graph $\pog$. \index{T-gain procedure!and derived graphs} Indeed:

\begin{thm}[\cite{TopologicalGraphTheory}]
Let $\pog$ be a gain graph, let $u$ be any vertex of $G$, and let $T$ be any spanning tree of $G$. Then the derived graph $G^{\bm_T}$ corresponding to $\langle G, \bm_T \rangle$ is isomorphic to the derived graph $G^{\bm}$. 
\label{thm:TGainIsomorphic}
\end{thm}
Note that this is a combinatorial rather than geometric result. 

\begin{proof}
This amounts to showing that there exists an appropriate relabeling of $G^{\bm}$. For each vertex $v$ of $G$, relabel the vertices $(v, z)$, $z \in \mathbb Z^d$ in the fiber over $v$ according to the rule $z \rightarrow z-c$, where $c$ is the net gain on the unique path from the root vertex $u$ to the vertex $v$. If $e$ is an edge originating at $v$, then we also change the indices of edges $(e, i)$ in the fiber over $e$ so that they agree with the relabeled indices of their initial points. This relabeling of vertices and edges defines an isomorphism $G^{\bm} \rightarrow G^{\bm_T}$.
\end{proof}

We say that the graphs $\pog$ and $\langle G, \bm_T \rangle$ are {\it $T$-gain related} and we write $\pog \sim \langle G, \bm_T \rangle$. More broadly, we say that $\pog$ and $ \langle G, \bm' \rangle$ are {\it $T$-gain equivalent} if $\pog \sim \langle G, \bm_T \rangle$ and $ \langle G, \bm' \rangle \sim \langle G, \bm_T \rangle$ for some choice of spanning tree $T$. In fact, if this is true for one spanning tree, it must be true for all choices of spanning tree, since the $T$-gain procedure preserves the net gains on cycles. $T$-gain equivalence can easily be shown to be an equivalence relation on the set of all gain assignments on a graph $G$. Theorem \ref{thm:TgainsPreserveRigidity} will demonstrate that $T$-gain equivalent graphs share the same generic rigidity properties. 

It is also possible to perform the $T$-gain procedure on subgraphs of $\pog$, by selecting a spanning tree for the subgraph and computing the $T$-gains on the edges of  $G$ accordingly. 

\begin{rem}
The $T$-gain procedure is not limited to the class of periodic graphs, but can be applied broadly  to any circumstance in which we are using an orbit graph. For example, the symmetric graphs discussed in the work of Schulze \cite{BS3,BSWWorbit} and others can be represented as gain graphs where the edges of the graph are labeled by elements of a symmetry group. The $T$-gain procedure will work in the same way in that case. 
\end{rem}

\subsection{The fixed $d$-torus, $\Tor^d$}
\label{sec:torus}

Let $\wt L$ be the $d \times d$ matrix whose rows are the linearly independent vectors $\{\bt_1, \dots, \bt_d\}$, $\bt_i \in \mathbb R^d$. Let $\wt L \mathbb Z^d$ denote the group generated by the rows of $\wt L$, viewed as translations of $\mathbb R^d$ (alternatively, we can think of this as the integer lattice, scaled by the rows of $\wt L$). We call $\wt L \mathbb Z$ the {\it fixed lattice}, and $\wt L$ is the {\it lattice matrix}. \index{lattice matrix} We call the quotient space $\mathbb R^d / \wt  L\mathbb Z^d$ the {\it  fixed $d$-torus} \index{fixed torus} generated by $\wt L$, and denote it $\Tor^d$. It follows that $a \equiv b$ in $\Tor^d$ if and only if $a-b = \sum_{i = 1}^d k_i\bt_i$, where $k_i \in \mathbb{Z}$.  

There is an equivalence class of sets of translations (equivalently, lattice matrices) which all generate the `same' torus, up to rotational orientation of the translation vectors at the origin. For any $d \times d$ matrix $\wt L$, there is a rotation matrix $R$ such that $R\wt L = L_0$, where $L_0$ is a lower triangular matrix.  $R$ is a $d \times d$ rotation matrix which rotates the parallelotope  generated by the rows of $\wt L$ ($d$-dimensional generalization of the parallelogram, see Coxeter \cite{RegularPolytopes}) such that $R\wt L = L_0$ is lower-triangular. 

We therefore assume, without loss of generality, that $\wt L$ is the lower triangular matrix $L_0$ 
$$L_0 = \left(\begin{array}{cccccc}t_{11} & 0  & 0 & \dots & 0 \\t_{12} & t_{22} & 0  & \dots & 0 \\ \vdots &  \vdots & \vdots & \dots & 0 \\t_{d1} & t_{d2} & t_{d3} &  \dots & t_{dd}\end{array}\right),$$
where $t_{\ell r} \in \mathbb R$ are the ${d+1 \choose 2}$ non-zero entries.

\begin{rem} It is possible to define a flexible lattice and flexible torus by simply allowing the entries of the lattice matrix to vary continuously with time. This is essentially the approach of Borcea and Streinu \cite{periodicFrameworksAndFlexibility}, and we will highlight extensions of the present work to the flexible torus setting where appropriate. Similarly,  it is possible to consider a range of partial flexibility by fixing some of the entries of the matrix and allowing others to vary. The `partially flexible $d$-torus' is treated in \cite{myThesis} and \cite{SymmetryPeriodic}, but will not be considered here. 

\end{rem}

\begin{rem}This representation of an abstract torus should not be confused with a realization of it. For example, we can realize the $2$-torus $\T^2$ in $\mathbb R^3$ as the familiar donut. This realization will change the metric properties of $\T^2$, due to the curvature of the surface in $\mathbb R^3$. However, $\T^2$ can also be realized in $\mathbb R^4$ in the following way:
\begin{eqnarray*}
\p: & \mathbb R^2 & \longrightarrow  \mathbb R^4\\
& (x, y) & \longrightarrow \frac{1}{2\pi}(\cos 2\pi x, \sin 2\pi x, \cos 2\pi y, \sin 2\pi y).
\end{eqnarray*} 
This is an isometric realization of $\T^2$ in $\mathbb R^4$, and it can be shown that this surface has zero Gaussian curvature everywhere, which explains why this realization is sometimes called the ``flat" torus. \index{flat torus} See \cite{doCarmo} or \cite{experiencingGeometry} for details. \qed \end{rem}

\section{Periodic frameworks}
\label{sec:periodicDefs}

The work of Borcea and Streinu on periodic frameworks is closely related to what is presented here. We will note, where appropriate, the connections and terminology that appear in their paper \cite{periodicFrameworksAndFlexibility}. It should be emphasized however that the work of the present paper was completed independently, as reflected in a 2009 talk at the sectional AMS meeting in Worcester \cite{Worcester}. 

At a general level, the work of Borcea and Streinu treats periodic frameworks as infinite simple graphs with periodic structure. In contrast, the work described here is concerned with finite frameworks on a torus, which {\it correspond} to infinite periodic frameworks. Both approaches share some common features: basic counting on orbit frameworks, a similar rigidity matrix, and fundamental results linking rigidity and infinitesimal rigidity. The two perspectives diverge on {\it genericity}. In \cite{periodicFrameworksAndFlexibility}, the edge directions are assumed to be generic, while in the present work, the edge directions are partially determined by the topology of the graph on the torus (the gains). In other words, we view this as part of the combinatorial information we seek to characterize. Only the positions of the vertices on the torus are assumed to be generic, as in finite rigidity.

\subsection{Periodic orbit frameworks on $\Tor^d$} 
\label{sec:periodicOrbitFrameworks}


Let $\Tor^d$ be the fixed $d$-torus  generated by a $d \times d$ matrix $\wt L$ (where $\wt L$ is not necessarily lower-triangular).  A {\it $d$-periodic orbit framework} \index{periodic orbit framework} is a pair $\pofw$,  where $\pog$ is a gain graph with gain group $\mathbb Z^d$, and $\p$ is an assignment of a unique geometric position on the fixed $d$-torus $\Tor^d$ to each vertex in $V$. That is, $\p: V \longrightarrow \Tor^d \subset \mathbb R^d$, with $\p(v_i) \neq \p(v_j)$ for $i \neq j$. We denote the position of the vertex $v_i$ by $\p(v_i) = \p_i$, and call $\p$ a {\it configuration} of $\pog$.  The geometric image of the edge $e = \{v_i, v_j; \bm_e\}$, is denoted $\{\p_i, \p_j + \bm_e\}$, and will be called a {\it bar} of the framework. The geometric vertices $\p_1, \dots \p_m$ will be called the {\it joints}. We will also call a $d$-periodic orbit framework $\pofw$ simply a {\it framework on $\Tor^d$}. \index{framework on $\Tor^d$!{\it see} {periodic orbit framework}} When we wish to talk only about the combinatorial structure of a periodic framework, we will refer to the gain graph $\pog$ as a {\it $d$-periodic orbit graph}. \index{periodic orbit graph}  Where it is clear from context we omit the `$d$'. 

The periodic framework $\pofw$ determines the {\it  derived periodic framework} \index{derived periodic framework} described by the pair $\dpfww$. The graph $G^{\bm} = (V^{\bm}, E^{\bm})$ is determined as described in Section \ref{sec:gainGraphs}, with the vertices and edges indexed by the elements of the integer lattice: $V^{\bm} = V \times \mathbb Z^d$, and  $E^{\bm} = E \times \mathbb Z^d$. The configuration $\p^{\bm}: V^{\bm} \rightarrow \mathbb R^d$ is determined by the configuration $\p$. The vertex $(v, z) \in V^{\bm}$ where $v \in V$, $z \in \mathbb Z^d$, has the following position:  
$$\p^{\bm}(v, z) = \p(v) + z \wt L,$$
where $\wt L$ is the lattice matrix whose rows are the generators of $\Tor^d$.

Similarly, from the derived periodic framework $\dpfww$ we can define the periodic framework $(\tg, \p)$. Let $G = (V, E)$ be the graph of vertices and edges consisting of all the elements of $G^{\bm}$ whose indices are the zero vector. The gain assignment $\bm$ is determined by the edges $E^{\bm}$. If, for example, the edge $(e, 0)$ connects vertices $(v_1, 0)$ and $(v_2, z)$ in $G^{\bm}$, then the directed edge $\{v_1, v_2\} \in E$ has gain $z$.

\subsection{$d$-periodic frameworks in $\mathbb R^d$}
\label{sec:periodicFrameworks}
In \cite{periodicFrameworksAndFlexibility}, Borcea and Streinu set out notation for the study of infinite graphs with periodic structure. They say that the pair $(\widetilde{G}, \Gamma)$ is a {\it $d$-periodic graph} \index{d-periodic graph@$d$-periodic graph} if $\widetilde{G}= (\widetilde V, \widetilde E)$ is a simple infinite graph with finite degree at every vertex, and $\Gamma \subset Aut(\widetilde{G})$ is a free abelian group of rank $d$, which acts without fixed points and has a finite number of vertex orbits. In other words, $\Gamma$ is isomorphic to $\mathbb Z^d$. 

Let $(\widetilde{G}, \Gamma)$ be a $d$-periodic graph, with $\widetilde{G} = (\widetilde V, \widetilde E)$. Borcea and Streinu  define a {\it  periodic placement} \index{periodic placement} \index{d-periodic framework@$d$-periodic framework!periodic placement} of $(\widetilde{G}, \Gamma)$ to be the pair $(\widetilde{\p}, \pi)$ given by the functions 
\[\widetilde{\p}:\widetilde{V} \rightarrow \mathbb R^d \ \ \ \ \textrm{and} \ \ \ \ \pi: \Gamma \rightarrow  Trans(\mathbb R^d),\]
where $\widetilde{\p}$ assigns positions in $\mathbb R^d$ to each of the vertices of $\widetilde{G}$, and $\pi$ is an injective homomorphism of $\Gamma$ into the group of translations of $\mathbb R^d$, denoted $Trans(\mathbb R^d)$. The image $\pi(\gamma)$ has the form $\pi(\gamma)(x) = x+\gamma^*$, where $\gamma^* \in \mathbb R^d$ is a translation vector. The placement functions $\widetilde{\p}$ and $\pi$ must satisfy
\[\p(\gamma v) = \pi(\gamma) (\p(v)), \]
or equivalently, 
\begin{equation}
\widetilde{\p}(\gamma v) = \widetilde{\p}(v) + \gamma^*.
\label{eqn:placement}
\end{equation}
Together, a $d$-periodic graph $(\widetilde{G}, \Gamma)$ and its periodic placement $(\widetilde{\p}, \pi)$ define a {\it $d$-periodic framework}, \index{d-periodic framework@$d$-periodic framework} which is denoted $(\widetilde{G}, \Gamma, \widetilde{\p}, \pi)$ \cite{periodicFrameworksAndFlexibility}. 



In contrast to the periodic orbit framework, the periodic framework has a countably infinite number of vertices and edges. The key relationship between these two different objects is the following: \index{d-periodic framework@$d$-periodic framework!as periodic orbit framework|(}
\begin{thm}
A $d$-periodic framework $\BS$ has a representation as the derived periodic framework $\dpfww$ corresponding to the periodic orbit framework $\pofw$ on $\Tor^d = \mathbb R^d / \wt L \mathbb Z^d$. 
\label{thm:representation}
\end{thm}
The proof of this result consists of picking representatives from the vertex orbits of the periodic framework $(\widetilde{G}, \Gamma, \widetilde{\p}, \pi)$, and using them to define the periodic orbit framework $\pofw$. We will describe this construction in detail, beginning with the following result about the graph $\BSg$.

The quotient multigraph $\widetilde G/\Gamma$ is finite since both $\widetilde V/\Gamma$ and $\widetilde E/\Gamma$ are finite \cite{periodicFrameworksAndFlexibility}. Let $G = \widetilde G/\Gamma$, and let $q_{\Gamma}: \widetilde G \rightarrow G$ be the quotient map. Then $q_{\Gamma}$ identifies each vertex orbit in $\widetilde G$ with a single vertex in $G$, and similarly for edges. 

\begin{thm}[Theorem 2.2.2 in \cite{TopologicalGraphTheory}]
Let $(\widetilde{G}, \Gamma)$ be a $d$-periodic graph, and let $G$ be the resulting quotient graph by the action of $\Gamma$. Then there is an assignment $\bm$ of gains in $\mathbb Z^d$ to the edges of $G$ and a labeling of the vertices of $\widetilde{G}$ by the elements of $V_G \times \mathbb Z^d$, such that $\widetilde{G} = G^{\bm}$ and the action of $\Gamma$ on $\widetilde{G}$ is the natural action of $\mathbb Z^d$ on $G^{\bm}$. 
\label{thm:periodicGraphAsDerivedGraph}
\end{thm}

\begin{proof} 
This follows directly from the proof of Theorem 2.2.2 in \cite{TopologicalGraphTheory}. See also \cite{myThesis}. 
\end{proof}

From Theorem \ref{thm:periodicGraphAsDerivedGraph}, we know that the $d$-periodic graph $(\widetilde{G}, \Gamma)$ can be described by the derived graph $G^{\bm}$ corresponding to a $d$-periodic orbit graph $\tg$, where $G = (V, E)$. We now show that there is also a correspondence between the periodic placement $(\widetilde{\p}, \pi)$ of $(\widetilde{G}, \Gamma)$  and the map $\p^{\bm}$ on $G^{\bm}$. 

Suppose that the generators of $\Gamma$ are given by $\{\gamma_1, \dots, \gamma_d\}$. Put 
\[\wt L = \left(\begin{array}{ccc} & \gamma_1^* &  \\ & \vdots &  \\ & \gamma_d^* & \end{array}\right), \textrm{\ where \ } \gamma_i^* \textrm{\ is determined by (\ref{eqn:placement}).}\]
Then $\wt L$ is the matrix whose rows are the translations of $\mathbb R^d$ that are the images under $\pi$ of the generators of $\Gamma$ (and again $\wt L$ is not necessarily lower-triangular). For $\gamma \in \Gamma$, let $z \in \mathbb Z^d$ be the row vector of coefficients of  $\gamma$ written as a linear combination of $\{\gamma_1, \dots, \gamma_d\}$. Then 
\begin{align*}
\widetilde{\p}(\gamma v) &  = \widetilde{\p}(v) + \gamma^*\\
							& = \widetilde{\p}(v) + z\wt L. 
\end{align*}

Let $A: \mathbb R^d \rightarrow \mathbb R^d$ be the linear transformation satisfying
 $$A(\gamma_i^*) = (\cdots ,0 , 1 , 0 , \cdots), $$
where the non-zero entry occurs in the $i$th column of the row vector. Then define 
\[A\wt L = \left(\begin{array}{ccc} & A\gamma_1^* &  \\ & \vdots &  \\ & A\gamma_d^* & \end{array}\right) = I_{d \times d}.\]
This permits us to write 
$$A\widetilde \p(\gamma v) = A \widetilde \p(v) + z \cdot A\wt L = A \widetilde \p(v) + z.$$

For each $v \in V$, there is exactly one vertex in $q_{\Gamma}^{-1}(v)$ (the orbit of $v$ in $\widetilde G$), whose image under $A\widetilde \p$ is in $[0,1)^d$. Label this vertex by $(v, 0)$, and label the other vertices in $q_{\Gamma}^{-1}(v)$ according to Theorem \ref{thm:periodicGraphAsDerivedGraph}. In addition, label the edges $e$ of $G$ by the same theorem, so that $\widetilde G = G^{\bm}$, the derived graph corresponding to $\tg$. 

To determine the map $\p^{\bm}: V^{\bm} \rightarrow \mathbb R^d$, for each $v \in V$, let $\p^{\bm}(v, 0) =  \widetilde \p(v, 0)$. For $a = (a_1, \dots, a_d) \in \mathbb Z^d$, let $\gamma_a = a_1\gamma_1 + \dots + a_d \gamma_d$. Now define 
\[	\p^{\bm}(v, a)  =  \widetilde \p (\gamma_a v).\]
Therefore, $A \p^{\bm}(v, a)=  A \widetilde \p(v) + a$, and applying the inverse linear transformation $A^{ -1}$, 
\[\p^{\bm}(v, a) = \widetilde \p(v) + a \wt L.\]
These observations complete the proof of Theorem \ref{thm:representation}.

Table \ref{table:graphNotation} summarizes in chart form the different graphs and notations for periodic frameworks just described. Because every $d$-periodic framework $\BS$ has a representation as the derived framework $\dpfww$ corresponding to the periodic orbit framework $\pofw$ on $\Tor^d$ (by Theorem \ref{thm:representation}), we adopt the following simplification of notation for $d$-periodic frameworks.  Let $\BS$ be an arbitrary periodic framework. Let $\wt L$ be the matrix described above, 
\[\wt L = \left(\begin{array}{ccc} & \gamma_1^* &  \\ & \vdots &  \\ & \gamma_d^* & \end{array}\right), \textrm{\ where \ } \gamma_i^* \textrm{\ is determined by (\ref{eqn:placement}).}\]
Then (\ref{eqn:placement}) can be rewritten
\[\wt p(v_i, z) = \wt p(v_i, 0) + z \wt L.\]


It is straightforward to show that:
\begin{prop}
A $d$-periodic framework $\BS$ is equivalent under rotation to a periodic framework $\pfwflr$ which is represented as the derived periodic framework $\dpfwfl$ corresponding to the periodic orbit framework $\pofw$ on $\Tor^d = \mathbb R^d / L_0 \mathbb Z^d$, where $L_0$ is lower triangular. 
\label{prop:rotationEquiv}
\end{prop}

As a consequence of this result, we assume that the configuration $\wt p$ in all subsequent frameworks $\pfwf$ is the rotated placement, and that $\Tor^d = \mathbb R^d / L_0\mathbb Z^d$, where $L_0$ is lower triangular.  \index{d-periodic framework@$d$-periodic framework!as periodic orbit framework|)} 

In the next section, we will explore to what extent the representation of $d$-periodic frameworks as $d$-periodic orbit graphs is unique. In addition, before we can define rigidity for periodic orbit frameworks, we first need to define length in this setting. 


\begin{verse}
\begin{center}
\begin{table}[h!]
\caption{Summary of notation for the different conceptions of periodic frameworks. \label{table:graphNotation}\index{periodic framework!notation}}
\begin{tabular}{|m{1.4cm}|m{3.8cm}|m{4.7cm}|m{3.5cm}|}
\hline
\begin{center} \it Graph \end{center} & \begin{center} \it Vertices \end{center}& \begin{center} \it Edges \end{center} & \begin{center} {\it Configuration} \end{center} \\
\hline

$$(\widetilde G, \Gamma)$$ $$G = \widetilde G / \Gamma$$ & $$\widetilde V, \ \  |\widetilde V| = \infty$$ $$ V = \wt V / \Gamma, \ \ |V|<\infty$$& $$\widetilde E, \ \  |\widetilde E| = \infty$$ $$E = \wt E / \Gamma, \ \ |E| < \infty$$ \center undirected edges & $$(\widetilde \p, \pi)$$ $$\widetilde \p: \widetilde V \rightarrow \mathbb R^d$$ $$\pi: \Gamma \rightarrow Trans (\mathbb R^d)$$\\
\hline

$$\tg$$ & $$\langle V, \bm \rangle = V$$  & $$\bm: E \rightarrow \mathbb Z^d$$ $$\langle E, \bm \rangle =$$ $$ \big\{\{v, w; \bm_e\}: \{v, w\} \in E\big\}$$ $$\{v, w; \bm_e\} = \{w, v; -\bm_e\} $$ $$|\langle E, \bm \rangle| = |E|$$ \center directed, labeled edges & $$\p: V \rightarrow \Tor^d$$ $$\Tor^d = \mathbb R^d / L_0 \mathbb Z^d$$ \\
\hline
$$\langle G^{\bm}, L_0 \rangle$$ & $$V^{\bm} =$$ $$ \{(v, z) :  v \in V, z \in \mathbb Z^d\}$$ $$|V^{\bm}| = \infty$$ & $$E^{\bm} =$$ $$\big\{\{e, z\}: e \in E, z \in \mathbb Z^d\big\}$$ $$ \{e, z\} = \{(v,z), (w,z+\bm_e)\} $$ $$|E^{\bm}| = \infty$$ & $$\p^{\bm}: V^{\bm} \rightarrow \mathbb R^d$$ $$\p^{\bm}(v, z) = \p(v) + zL_0$$ \\
\hline
\end{tabular}
\end{table}
\end{center}
\end{verse}

\subsection{Equivalence relations among $d$-periodic orbit frameworks}
\label{sec:equivRelations}

We now define notions of length and congruence for frameworks on $\Tor^d$, which leads to an equivalence relation among all $d$-periodic orbit graphs. Let $L_0$ be the lower triangular matrix whose rows are the translations $\{\bt_1, \dots, \bt_d\}$, where $\Tor^d = \mathbb R^d /  L_0 \mathbb Z^d$.

Given an edge $e = \{v_i, v_j; \bm_e\} \in E\pog$, we define the {\it  length} \index{periodic orbit framework!edge length} of the edge $e$ to be the Euclidean length of the vector $(\p_i - (\p_j + \bm_eL_0))$, denoted $\|\p_i - (\p_j + \bm_eL_0)\|$.  

More generally, for any pair of joints $p_i, p_j$ and any element $m_{ij} \in \mathbb Z^d$, we write $\|\{p_i, p_j; m_{ij}\}\|$ to denote the Euclidean distance of the vector $(\p_i - (\p_j + \bm_{ij}L_0))$. Note that this need not be the same as $\|\{p_j, p_i; m_{ij}\}\|$. That is, the order of the vertices matters. 
		
By definition, the edges of $\dpfwo$, \index{derived periodic framework!edge length} have the same lengths as the edges of $(\tg, \p)$. Let $e = \{v_i, v_j; \bm_e\}$ be an edge of $\tg$. The edge $(e, z) \in E^{\bm}$ connects the vertex $(v_i, z)$ to the vertex $(v_j, \bm_e + z)$. Hence 		\begin{align*}
\|(e, z)\| = &  \|(\p_i+zL_0) - (\p_j + \bm_eL_0+zL_0)\|\\
		=	& \|\p_i - (\p_j + \bm_eL_0)\|\\
		=	& \|e\|.
		\end{align*}
In other words, all edges in the fibre over $e$ have length $\|e\|$. 

Now let $L_0 = I_{d \times d}$ be the $d$-dimensional identity matrix, and consider $\Unit^d = [0,1)^{d|V|}$ to be the unit torus generated by $L_0$.
Let $\p_i =(p_{i1}, \dots, p_{id}) \in \mathbb R^d$. We write $\lfloor \p_i \rfloor$ to denote the vector $(\lfloor p_{i1} \rfloor, \dots, \lfloor p_{id} \rfloor)$, where $\lfloor x \rfloor, x \in \mathbb R$ is the {\it floor function}, defined to be the largest integer less than or equal to $x$. We say that the framework $(\tgn, \bq)$ is {\it  $\Unit^d$-congruent} to $(\tg, \p)$  if there exists a vector $t \in \mathbb R^d$ such that \index{periodic orbit framework!congruent}
\begin{enumerate}[(a)]
\item $\bq_i = (\p_i + t) - \lfloor \p_i+t \rfloor$  for each vertex $v_i \in V$, and
\item $\bn_e = \bm_e + (\lfloor \p_j + t \rfloor - \lfloor \p_i + t \rfloor)$.
\end{enumerate}
We write $ (\tgn, \bq) \cong \pofw$.

If $(\tg, \p)$ and $(\tg, \bq)$ are two periodic frameworks with the same underlying gain graph $\tg$, the description of congruence is more simple. In this case (b) is automatically satisfied, and (a) becomes simply $\bq_i = \p_i + t$, for all $v_i \in V$.

More generally, if $\Tor^d$ is the torus generated by the matrix $L_0$, then there is an affine transformation mapping $L_0$ to the $d \times d$ identity matrix. We say that  the framework $(\tgn, \bq)$ is {\it  $\Tor^d$-congruent} to $(\tg, \p)$  if their corresponding affine images on the unit torus are  $\Unit^d$-congruent. With these definitions in place, it is easy to confirm the following:

\begin{prop} $\Tor^d$-congruence is an equivalence relation on the set of all periodic frameworks. 
\label{prop:congEquiv}
\end{prop}

We say that the gain graphs $\tg$ and $\tgn$ are {\it  periodic equivalent} \index{periodic orbit graph!periodic equivalent} if there exist configurations $\p$ and $\bq$ such that the periodic frameworks $(\tg, \p)$ and $(\tgn, \bq)$ are $\Tor^d$-congruent. Following directly from Proposition \ref{prop:congEquiv}, we have:

\begin{prop}
Periodic equivalence is an equivalence relation on the set of all $d$-periodic orbit graphs. 
\end{prop} 

For two periodic equivalent graphs $\tg$ and $\tgn$, the net gain on any cycle is the same. \index{periodic orbit graph!net cycle gain}
For any vertex $v \in V$, let $\ell(v_i) = \lfloor \p_i+ t \rfloor$, where $\p$ is the configuration such that $(\tg, \p) \cong (\tgn, \bq)$ for some configuration $\bq$ of $\tgn$.  
Consider a cycle $C$ of edges in $G$. The net gain on $C$ in $\tg$ is 
\[\sum_{e \in C}\bm_e.\]
In the graph $\tgn$, the same cycle has gain 
\begin{align}
\sum_{e \in C} \bn_e & = \sum_{e \in C} (\bm_e + \ell(t(e)) - \ell(o(e))) \nonumber \\
 						& =  \sum_{e \in C} \bm_e + \sum_{e \in C}\ell(t(e)) - \sum_{e \in C}\ell(o(e)) 
\label{eqn:cycleSum}
\end{align}
where we denote the origin of the directed edge $e$ by $o(e)$, and the terminus by $t(e)$. Since $C$ is a cycle, each vertex appears exactly once as the origin of an edge, and exactly once as the terminus of another edge. Hence the last two sums in (\ref{eqn:cycleSum}) cancel, and we obtain
\[\sum_{e \in C}\bm_e = \sum_{e \in C}\bn_e.\]
The following proposition follows from these observations:

\begin{prop}
If $\tg$ and $\tgn$ are periodic equivalent, they have the same gain space: \index{periodic orbit graph!gain space}
\[\gs\pog = \gs\pogn.\]
\end{prop}

\begin{rem}
All of the definitions and results of the previous sections are also sensible for describing frameworks on a flexible torus. Since frameworks are defined at a particular moment in time, no changes are required to the definitions. Simply replace the matrix $L_0$ by a matrix of variables. 
\end{rem}

\section{Rigidity and infinitesimal rigidity on $\Tor^d$}
\label{sec:rigidity}
\subsection{Rigidity on $\Tor^d$}
\label{sec:periodicRigidity}

 Let $\pofw$ be a periodic orbit framework with $m: E \rightarrow \mathbb Z^d$ and $\p:V \rightarrow \Tor^d = \mathbb R^d / L_0 \mathbb Z^d$, where $V = \{v_1, v_2, \dots, v_{n}\}$. A {\it motion} of the framework on $\Tor^d$ is an indexed family of functions $P_i: [0,1] \rightarrow \mathbb R^d$, $i = 1, \dots, |V|$ such that: \index{motion!of $\pofw$ on $\Tor^d$}
\begin{enumerate}
	\item $P_i(0) = p(v_i)$ for all $i$;
	\item $P_i(t)$ is continuous on $[0,1]$, for all $i$; 
	\item For all edges $e = \{v_i, v_j; \bm_e \} \in E\pog$, 
	\[\|P_i(t) - (P_j(t)+\bm_eL_0)\| = \|p(v_i) - (p(v_j) + \bm_eL_0) \|\] for all $t \in [0,1]$.
\end{enumerate}

In other words, a motion $P_i$ of a periodic orbit framework $\pofw$ \index{periodic orbit framework!motion on $\Tor^d$} preserves the distances between each pair of vertices connected by an edge. Let $M = \{-1, 0, 1\}$. Let $M^d$ represent the set of all $d$-tuples with entries from the set $M$. If a motion $P_i$ preserves all of the distances $\|\{p_i, p_j; m\}\|$, where $v_i, v_j \in V$, and $m \in M^d$, then we say that $P_i$ is a {\it rigid motion} or {\it trivial motion}. \index{motion!trivial} Note that there will be some duplication among this set of distances, for example, $\|\{p_i, p_j; m_{\alpha}\}\| = \|\{p_j, p_i; -m_{\alpha}\}\|$, which we could eliminate with further restrictions on $m$. 
\begin{prop}
Given any pair of vertices $v_i, v_j$,  a rigid motion preserves the length of the segment $\|\{p_i, p_j; m\}\|$ for all $m \in \mathbb Z^d$.
\end{prop}

If the only motions of a framework $\pofw$ on $\Tor^d$ are rigid motions, then we say that the framework $\pofw$ is {\it rigid on the fixed torus $\Tor^d$}. \index{periodic orbit framework!rigid} 

%
%

\subsection{Infinitesimal rigidity of frameworks on $\Tor^d$}
\label{sec:periodicInfinitesimalRigidity}


 An {\it infinitesimal motion} \index{periodic orbit framework!infinitesimal motion on $\Tor^d$} of a periodic orbit framework $\pofw$ on $\Tor^d$ is an assignment of velocities to each of the vertices, $\bu: V \rightarrow \mathbb{R}^d$, with $u(v_i) = u_i$ such that 
 \begin{equation}
(\bu_i - \bu_j)\cdot(\p_i - \p_j-\bm_eL_0) = 0
\label{eqn:fixTorMot}
\end{equation}
for each edge $e = \{v_i, v_j;\bm_e\} \in E\pog$. Such an infinitesimal motion preserves the lengths of any of the bars of the framework (see Figure \ref{fig:motions}). 

A {\it trivial infinitesimal motion} of $\pofw$ on $\Tor^d$ is an infinitesimal motion that preserves the distance between all pairs of vertices:
\begin{equation}
(\bu_i - \bu_j)\cdot(\p_i - \p_j-\bm_eL_0) = 0
\label{eqn:fixTorTrivMot}
\end{equation}
for all triples  $\{v_i, v_j;\bm_e\}$, $\bm \in \mathbb Z^d$.  For any periodic orbit framework $\pofw$ on $\Tor^d$, there will always be a $d$-dimensional space of  trivial infinitesimal motions of the whole framework, namely the space of infinitesimal translations.  See Figure \ref{fig:translations}.
\begin{verse}
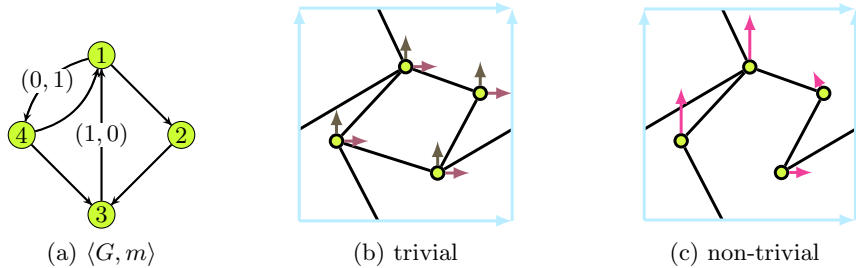
\begin{figure}
\begin{center}
\subfloat[$\pog$]{\label{fig:gainGraphMotions}\begin{tikzpicture}[auto, node distance=1.5cm]
\tikzstyle{vertex1}=[circle, draw, fill=couch, inner sep=1pt, minimum width=3pt, font=\footnotesize]; 
\tikzstyle{gain} = [fill=white, inner sep = 0pt,  font=\footnotesize, anchor=center];

\node[vertex1] (1) {$1$}; 
\node[vertex1] (2) [below right of=1] {$2$}; 
\node[vertex1] (3) [below left of=2] {$3$}; 
\node[vertex1] (4) [below left of=1] {$4$}; 

\path[thick] (1) edge (2) 
(2) edge (3) 
(3) edge   (1) 
(1) edge [bend right]  (4)
(4) edge [bend right]   (1)
(4) edge  (3); 

\pgfsetarrowsend{stealth}[ shorten >=2pt]
\path (1) edge  (2) 
(2) edge  (3) 
(3) edge node[gain] {$(1,0)$} (1) 
(1) edge [bend right] node[gain] {$(0,1)$} (4)
(4) edge [bend right]   (1)
(4) edge  (3); 
\pgfsetarrowsend{}
\end{tikzpicture}}\hspace{.5in}%
\subfloat[trivial]{\label{fig:translations}\includegraphics[width=1.2in]{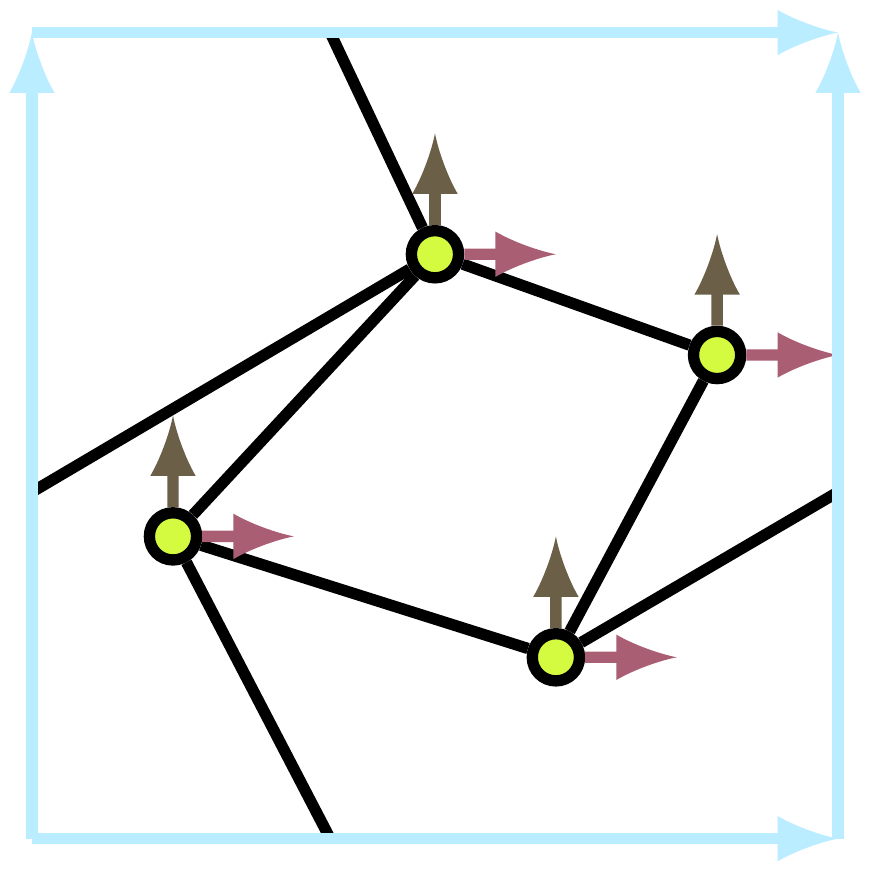}} \hspace{.5in}
\subfloat[non-trivial]{\label{fig:nontrivialMotion} \includegraphics[width=1.2in]{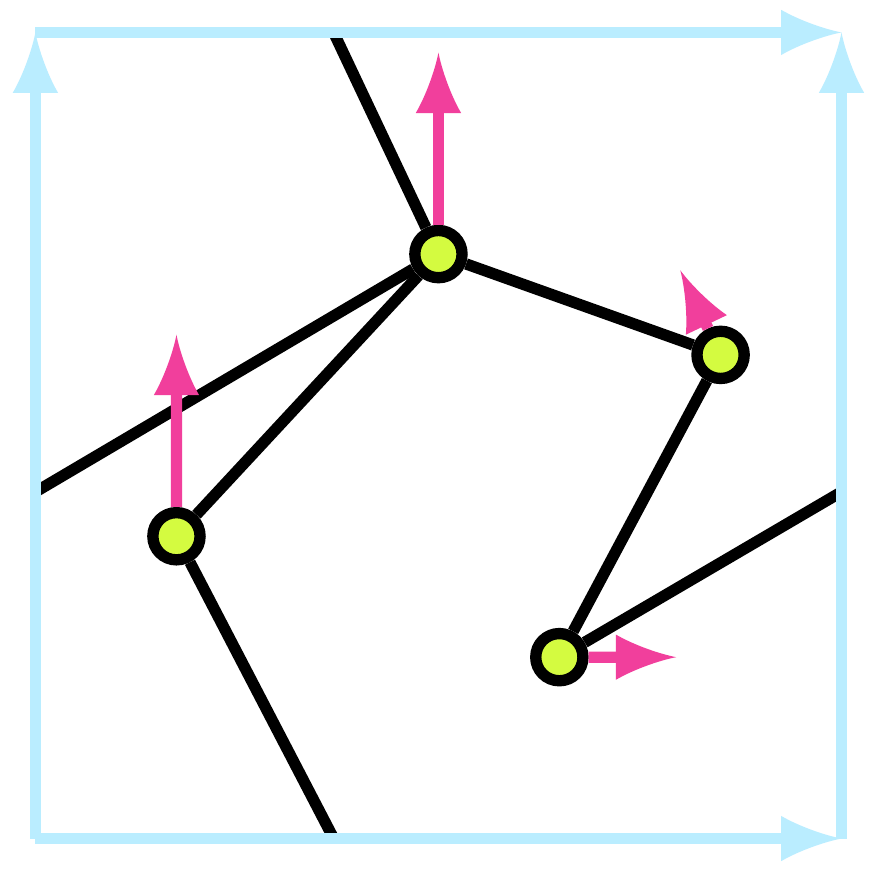}}
\caption{ A periodic orbit framework (a). Two trivial infinitesimal motions (translations) for a framework on $\Tor^2$ are indicated in (b). Removing a single edge produces a non-trivial infinitesimal motion on the modified framework pictured in (c). \label{fig:motions}}
\end{center}
\end{figure}\end{verse}

Rotation is not a trivial motion for periodic orbit frameworks thus defined, because a rotation of a graph on $\Tor^d$ will always change the distance between some pair of points. This is a consequence of the fact that we have fixed our representation of $\Tor^d$, and are considering motions of the periodic orbit framework relative to the fixed torus. 
This is in contrast to the approach of Borcea and Streinu, who do view rotations as trivial infinitesimal motions of the infinite framework $\BS$ in $\mathbb R^d$. Recall that our frameworks on the torus are equivalence classes of the periodic frameworks $\BS$, where two such frameworks are equivalently represented by the orbit framework $\pofw$ if they are rotations of one another in $\mathbb R^d$. \index{periodic orbit framework!trivial infinitesimal motion}

\begin{prop}
If $u$ is a trivial infinitesimal motion of $\pofw$ on $\Tor^d$, then $u$ is an infinitesimal translation. 
\end{prop}

\begin{proof} 
Let $u = (u_1, \dots, u_{|V|})$ be an infinitesimal motion satisfying (\ref{eqn:fixTorTrivMot}) for all values of $m_e$ of the form $(0, \dots, 0, 1, 0, \dots, 0)$ (vectors of $0$'s with a single $1$ in the $i$-th place).  Elementary linear algebra demonstrates that the simultaneous solution of 
\[(\bu_i - \bu_j)\cdot(\p_i - \p_j-\bm_eL_0) = 0\]
for all such values of $m_e$ will yield the single solution, $u_1, = u_2 = \cdots = u_{|V|}$, which corresponds to an infinitesimal translation. 
\end{proof}

If the only infinitesimal motions of a framework $\pofw$ on $\Tor^d$ are trivial (i.e. infinitesimal translations), then it is {\it infinitesimally rigid}. Otherwise, the framework is {\it infinitesimally flexible}. \index{periodic orbit framework!infinitesimal rigidity on $\Tor^d$} \index{infinitesimal motion!of $\pofw$}

 An infinitesimal motion $u$ of $\pofw$ on $\Tor^d$ is called an {\it infinitesimal flex} if
 \begin{equation}
(\bu_i - \bu_j)\cdot(\p_i - \p_j-zL_0) \neq 0
\end{equation}
for some triple $\{v_i, v_j; z\}$, where $v_i, v_j \in V$, and $z \in \mathbb Z^d$. Note that we no longer require that the vertices of the framework affinely span $\mathbb R^d$, in contrast to the analogous definition for finite frameworks (see \cite{CountingFrameworks} for example). This is a consequence of the fact that we need only find {\it some} triple $\{v_i, v_j; z\}$ for which $(\bu_i - \bu_j)\cdot(\p_i - \p_j-zL_0) \neq 0$, and we are free to choose $z$ from $\mathbb Z^d$. 

\begin{ex}
\label{ex:4vertex}
The framework on $\Tor^2$ shown in Figure \ref{fig:translations} is infinitesimally rigid. The only infinitesimal motions of this framework are trivial, as indicated. 
Removing a single bar $(\p_3, \p_4; (0,0))$ from the orbit graph shown in (a) yields a framework with a non-trivial infinitesimal motion (a flex). Figure \ref{fig:nontrivialMotion} depicts this motion, which was found by solving the rigidity matrix described below. \qedhere
\end{ex}

\subsection{Infinitesimal rigidity of  $\pfwf$ in $\mathbb R^d$}

We now confirm that the representation of $\pfwf$ as an orbit framework on the torus provides us with the information we seek, namely the infinitesimal motions of $\pfwf$ in $\mathbb R^d$ that preserve its periodicity.

An infinitesimal periodic motion of $\pfwf$ in $\mathbb R^d$ is a function $\wt u: \wt V \rightarrow \mathbb R^d$ such that the infinitesimal velocity of every vertex of $\wt G$ in an equivalence class (under $\mathbb Z^d$) is identical. Recall that $\pfwf$ can be represented as $\dpfwo$, and therefore the vertices of $\wt G = G^m$ are naturally indexed by the elements of $\mathbb Z^d$. Then an {\it infinitesimal periodic  motion} of $\pfwf$ in $\mathbb R^d$ is a function $\wt u: \wt V \rightarrow  \mathbb R^d$ such that the following two conditions are satisfied: \index{periodic framework!infinitesimal per. motion}
\begin{enumerate}
	\item For every edge $e = \{(v_i, a), (v_j, b)\} \in \wt E$, 
	\[\big(\wt p(v_i, a) - \wt p(v_j, b) \big) \cdot \big (\wt u(v_i, a) - \wt u(v_j, b) \big) = 0, \]
	\item $\wt u(v_i, z) = \wt u(v_i, 0)$, for all $z \in \mathbb Z^d$.
\end{enumerate}
The framework $\pfwf$ is {\it  infinitesimally periodic rigid} in $\mathbb R^d$ if the only such motions assign the same infinitesimal velocity to all vertices of $\wt V$ (i.e. they are translations). \index{periodic framework!infinitesimal per. rigidity}

\begin{rem}
An infinitesimal motion of $\pfwf$ is a motion that is itself periodic in that $\wt u$ assigns the {\it same} infinitesimal velocity to every vertex in an equivalence class. It is possible to relax this assumption to consider infinitesimal motions that preserve the periodicity of the framework, but that are not themselves periodic, since they also change the lattice. These motions correspond to motions of the periodic orbit framework on the {\it flexible} torus. \qed
\end{rem}

\begin{prop}
Let $\pfwf$ be a $d$-periodic framework.  Let $\pofw$ be its $d$-periodic orbit framework given by Proposition \ref{thm:representation}. Then the following are equivalent:
\begin{enumerate}[(i)]
	\item $\pfwf$ is infinitesimally periodic rigid in $\mathbb R^d$
	\item $\pofw$ is infinitesimally rigid on $\Tor^d = \mathbb R^d / L_0 \mathbb Z^d$.	
\end{enumerate}
\label{prop:equivRepsFixed}
\end{prop}

\begin{proof}
Let $u$ be an infinitesimal motion of $\pofw$ on $\Tor^d$. We extend $u$ to an infinitesimal motion $\wt u$ of $\pfwf = \dpfwo$ by letting every vertex of $\dpfwo$ in the fibre over $v \in V\pog$ have the same infinitesimal velocity. More precisely, let 
\[\wt u(v, z) = u(v), \ \forall z \in \mathbb Z^d.\] 
Since an edge  $ (e, a) = \{(v_i, a), (v_j, b)\} \in \wt E$ if and only if $e = \{v_i, v_j; b-a\} \in E\pog$, the fact that $\wt u$ is an infinitesimal periodic motion of $\pfwf$ is obvious. 

On the other hand, given an infinitesimal motion $\wt u$ of $\pfwf$, let $u:V \rightarrow \mathbb R^d$ be given by 
\[u(v_i) = \wt u(v_i, 0).\]
Again it is clear that $u$ is an infinitesimal motion of $\pofw$ on $\Tor^d$. 

In both cases, the non-trivial motions assign the same velocities to all vertices of $\pofw$ or $\pfwf$ respectively, and therefore non-trivial infinitesimal motions of $\pofw$ on $\Tor^d$ correspond to non-trivial infinitesimal periodic motions of $\pfwf$ in $\mathbb R^d$.%
\end{proof}

\begin{rem}
Proposition \ref{prop:equivRepsFixed} also holds when we replace ``infinitesimally rigid" with ``rigid". Because our focus is infinitesimal rigidity, we omit the statement and proof of this version. 
\qed \end{rem}

\begin{rem}
The reader should be reminded that an infinite framework $(\wt G, \wt p)$ may be infinitesimally periodic rigid without being infinitesimally rigid, since there may be non-trivial infinitesimal motions of the framework that do not preserve the periodicity. Hence it is important to distinguish between these forms of rigidity, and we emphasize that we are interested in forced periodicity, not incidental periodicity. \qed \end{rem}

If $\pofw$ is infinitesimally rigid on $\Tor^d$, then $\pofw$ is rigid on $\Tor^d$. Or, in other words, if a framework is flexible, then it also has an infinitesimal flex.  A periodic-adapted proof of this fact using the averaging technique is presented in Section \ref{sec:averaging}, after the definition of the rigidity matrix. The converse is not true,  as illustrated in the example pictured in Figure \ref{fig:infFlexNotFiniteFlex}. However, geometrically this example is highly `special'. It is known that for {\it generic} frameworks (defined in Section \ref{sec:genericFrameworks}), infinitesimal rigidity and rigidity actually coincide. This is a periodic analogue of a well-known result due to Asimow and Roth \cite{AsimowRoth} in the theory of rigidity for finite graphs.

\begin{verse}
\begin{figure}[htbp]
\begin{center}
\subfloat[$\pog$]{\begin{tikzpicture}[auto, node distance=2cm, thick]
\tikzstyle{vertex1}=[circle, draw, fill=couch, inner sep=1pt, minimum width=3pt, font=\footnotesize]; 
\tikzstyle{vertex2}=[circle, draw, fill=melon, inner sep=1pt, minimum width=3pt, font=\footnotesize]; 

\tikzstyle{gain} = [fill=white, inner sep = 0pt,  font=\footnotesize, anchor=center];

\node[vertex2] (1) {$1$}; 
\node[vertex1] (2) [below right of=1] {$2$}; 
\node[vertex1] (4) [below left of=1] {$3$}; 

\path[thick] (1) edge [bend right] (2) 

(4) edge [bend right]   (1)
(4) edge (2); 

\pgfsetarrowsend{stealth}[ shorten >=2pt]
\path[thick] (1) edge [bend left] node[gain] {$(1,0)$} (2) 
(1) edge [bend right] node[gain] {$(-1,0)$} (4)
(4) edge [bend right] node[gain] {$(1,0)$}  (2);
\pgfsetarrowsend{}
\end{tikzpicture}}\hspace{1cm}%
\subfloat[$\pofw$]{\begin{tikzpicture}[thick,scale=1] 
\tikzstyle{vertex1}=[circle, draw, fill=couch, inner sep=0pt, minimum width=5pt]; 
\tikzstyle{vertex2}=[circle, draw, fill=melon, inner sep=0pt, minimum width=5pt]; 

	\pgftransformscale{1.5} 
	
		\draw[thick] (-1,0) -- (3, 0);
	
	\node[vertex2] (v1) at (0, .15){};
	\node[vertex1] (v3) at (-.5, 0){};
	\node[vertex1] (v2) at (0.5, 0) {};

	\node[vertex2] (v4) at (2, .15){};
	\node[vertex1] (v5) at (1.5, 0){};
	\node[vertex1] (v6) at (2.5, 0) {};

	\draw (v3) -- (v1) -- (v2) (v6) -- (v4) -- (v5) (v1) -- (v5) (v4) -- (v2); 
	\draw (v1) -- (-1, .05) (v4) -- (3, .05);
	
	\pgfsetarrowsend{latex} 
	\draw[pink, very thick] (v1) -- (0, 0.5);
	\draw[pink, very thick] (v4) -- (2, 0.5);
	\pgfsetarrowsend{}

	\draw[sky, very thick] (-1, -1) -- (1, -1);
	\draw[sky, very thick] (-1, 1) -- (1, 1);
	\draw[sky, very thick] (-1, -1) -- (-1, 1);
	\draw[sky, very thick] (1, -1) -- (1, 1);
	\pgftransformxshift{2cm}
	\draw[sky, very thick] (-1, -1) -- (1, -1);
	\draw[sky, very thick] (-1, 1) -- (1, 1);
	\draw[sky, very thick] (-1, -1) -- (-1, 1);
	\draw[sky, very thick] (1, -1) -- (1, 1);

\end{tikzpicture} }

\caption{The framework $\pofw$ has a infinitesimal flex on $\Tor^2$ (b), but no finite flex. The position of the vertices of $\pofw$ has all three vertices on a line, however, the drawing has been exaggerated to indicate the connections between vertices in adjacent cells. }
\label{fig:infFlexNotFiniteFlex}
\end{center}
\end{figure}
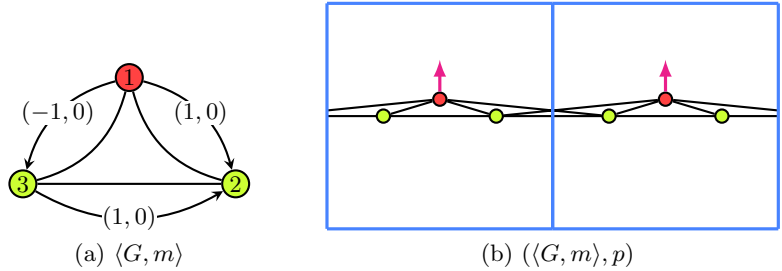\end{verse}

\subsection{The fixed torus rigidity matrix}
\label{sec:theRigidityMatrix}

Rigidity matrices for periodic frameworks have been recorded by Guest and Hutchinson \cite{DeterminancyRepetitive}, Borcea and Streinu \cite{periodicFrameworksAndFlexibility}, and Malestein and Theran \cite{Theran}. The matrix we present below is different from these other presentations, for two reasons. The first is that this is the matrix for the fixed torus, and the second is that we are considering equivalence classes of frameworks under rotation. 

The {\it rigidity matrix}, $\R_0\pofw$,  \index{rigidity matrix!fixed torus} \index{rigidity matrix!periodic orbit framework} records equations for the space of possible infinitesimal motions of a $d$-periodic orbit framework. It is the $|E| \times d|V|$ matrix with one row of the matrix corresponding to each edge $e = \{i, j); m_e\}$ of $\pog$ as follows: 
\[  \renewcommand{\arraystretch}{1.2}
 \bordermatrix{ &   & i &   & j &    \cr
 & & & \vdots & & \cr
\textrm{edge } \{i, j; m_e\} & 0 \cdots 0 &\p_i - (\p_j + \bm_eL_0) & 0 \cdots 0 & (\p_j  + \bm_eL_0) - \p_i & 0 \cdots 0 \cr
 & & & \vdots & & \cr},\]
where each entry is actually a $d$ -dimensional vector, and the non-zero entries occur in the columns corresponding to vertices $v_i$ and $v_j$ respectively.
By construction, the kernel of this matrix will be the space of infinitesimal motions of $\pofw$ on $\Tor^d$. By an abuse of notation we may write $$\R_0\pofw \cdot \bu^T = 0$$ 
where $\bu = (\bu_1, \bu_2, \dots, \bu_{|V|})$, and $\bu_i \in \mathbb R^d$.  That is, $u$ is an infinitesimal motion of the joints of $\pofw$ on $\Tor^d$.

\begin{ex} Consider the periodic orbit graph $\pog$ shown in Figure \ref{fig:gainGraphMotions}. Let $L_0$ be the matrix generating the torus $\Tor^2$. The rigidity matrix $\R_0\pofw$ will have have six rows, and eight columns (two columns corresponding to the two coordinates of each vertex). 

\[ \bordermatrix{
&  \p_1  & \p_2 & \p_3 & \p_4 \cr
\{1, 2; (0, 0)\} & \p_1 - \p_2 & \p_2 - \p_1 &   0 &   0 \cr
\{2, 3; (0, 0)\} &   0 & \p_2 - \p_3 & \p_3 - \p_2 &   0 \cr
\{3, 4; (0, 0)\} &   0 &   0 & \p_3 - \p_4 & \p_4 - \p_3 \cr
\{1, 4; (0, 0)\} & \p_1 - \p_4 &   0 &   0 & \p_4 - \p_1 \cr  
\{1, 3; (-1, 0)\} & \p_1 - \p_3 + (1, 0)L_0 &   0 & \p_3 - \p_1 - (1, 0)L_0&   0 \cr  
\{1, 4; (0, 1)\} & \p_1 - \p_4 - (0,1)L_0 &   0 &   0 & \p_4 - \p_1+(0, 1)L_0   
} \] 

\qed
\label{ex:rigMatrix}
\end{ex}

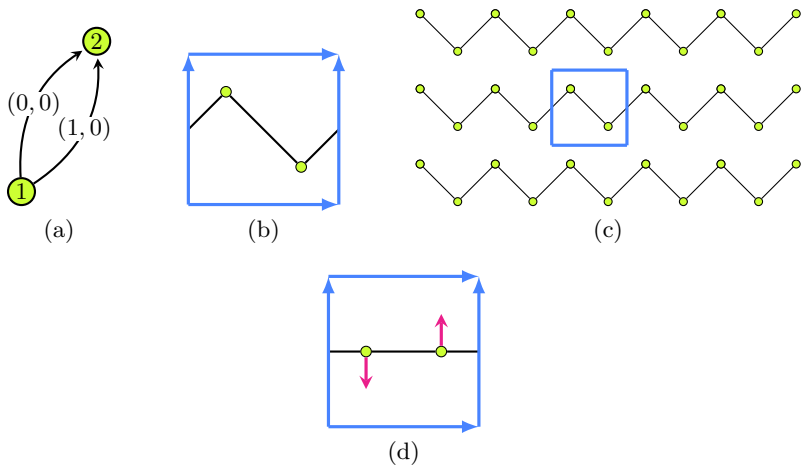
\begin{figure}
\begin{center}
\subfloat[]{\begin{tikzpicture}[->,>=stealth,shorten >=1pt,auto,node distance=2.8cm,thick, font=\footnotesize] 

\tikzstyle{vertex1}=[circle, draw, fill=couch, inner sep=1pt, minimum width=3.5pt]; 
\tikzstyle{vertex2}=[circle, draw, fill=lips, inner sep=1pt, minimum width=3.5pt]; 
\tikzstyle{vertex3}=[circle, draw, fill=melon, inner sep=1pt, minimum width=3.5pt]; 
\tikzstyle{vertex4}=[circle, draw, fill=bluey, inner sep=1pt, minimum width=3.51pt]; 
\tikzstyle{gain} = [fill=white, inner sep = 0pt,  font=\footnotesize, anchor=center];

\node[vertex1] (1) at (0, 0){$1$}; 
\node[vertex1] (2) at (1, 2) {$2$};

\path (1) edge [bend right] node[gain] {$(1,0)$} (2);
\path (1) edge [bend left] node[gain] {$(0,0)$} (2);

\end{tikzpicture}} \hspace{1cm}%
%
\subfloat[]{\begin{tikzpicture} 
\tikzstyle{vertex1}=[circle, draw, fill=couch, inner sep=1pt, minimum width=4pt]; 
\tikzstyle{vertex2}=[circle, draw, fill=lips, inner sep=1pt, minimum width=4pt]; 
\tikzstyle{vertex3}=[circle, draw, fill=melon, inner sep=1pt, minimum width=4pt]; 
\tikzstyle{vertex4}=[circle, draw, fill=bluey, inner sep=1pt, minimum width=4pt]; 

\node[vertex1] (1) at (1,-1) {};
\node[vertex1] (2) at (0,0) {};

\draw[thick] (1) -- (2)  ;
\draw[thick] (1) -- (1.5, -0.5);
\draw[thick] (-0.5, -0.5) -- (2);

\pgfsetarrowsend{latex} 
	\draw[sky, very thick] (-.5, -1.5) -- (1.5, -1.5);
	\draw[sky, very thick] (-.5, .5) -- (1.5, .5);
	\draw[sky, very thick] (-.5, -1.5) -- (-.5, .5);
	\draw[sky, very thick] (1.5, -1.5) -- (1.5, .5);
\pgfsetarrowsend{} 
\end{tikzpicture}}\hspace{1cm}%
\subfloat[]{\begin{tikzpicture} 
\tikzstyle{vertex1}=[circle, draw, fill=couch, inner sep=1pt, minimum width=3pt]; 
\tikzstyle{vertex2}=[circle, draw, fill=lips, inner sep=1pt, minimum width=3pt]; 

\foreach \x in {-2, -1, 0,1,2} 
\foreach \y in { -1, 0,1} 
{ 
\node[vertex1] (1\x\y) at (\x+.5, \y-.5){};
} 

\foreach \x in {-2, -1, 0,1,2} 
\foreach \y in {-1, 0,1} 
{ 
\node[vertex1] (2\x\y) at (\x,\y) {}; 
} 
\draw \foreach \x in { -2,-1, 0,1, 2} 
\foreach \y in {-1, 0,1} {(1\x\y) -- (2\x\y)};

\draw \foreach \x in { -2,-1, 0,1,2} 
\foreach \y in {-1, 0,1} {(1\x\y) -- (\x+1, \y) };

Redrawing vertex 1 for prettiness
\foreach \x in {-2, -1, 0,1,2,3} 
\foreach \y in { -1, 0,1} 
{ 
\node[vertex1] (2\x\y) at (\x, \y) {};
} 

	\draw[sky, very thick] (-.25, -.75) -- (.75, -.75);
	\draw[sky, very thick] (-.25, .25) -- (.75, .25);
	\draw[sky, very thick] (-.25, -.75) -- (-.25, .25);
	\draw[sky, very thick] (.75, -.75) -- (.75, .25);

\end{tikzpicture}}\hspace{1cm}%
\subfloat[]{\begin{tikzpicture} 
\tikzstyle{vertex1}=[circle, draw, fill=couch, inner sep=1pt, minimum width=4pt]; 
\tikzstyle{vertex2}=[circle, draw, fill=lips, inner sep=1pt, minimum width=4pt]; 
\tikzstyle{vertex3}=[circle, draw, fill=melon, inner sep=1pt, minimum width=4pt]; 
\tikzstyle{vertex4}=[circle, draw, fill=bluey, inner sep=1pt, minimum width=4pt]; 

\node[vertex1] (1) at (1,-.5) {};
\node[vertex1] (2) at (0,-.5) {};

\draw[thick] (1) -- (2)  ;
\draw[thick] (1) -- (1.5, -0.5);
\draw[thick] (-0.5, -0.5) -- (2);

\pgfsetarrowsend{stealth}
\draw[pink, very thick] (1) -- (1,0);
\draw[pink, very thick] (2) -- (0, -1);
\pgfsetarrowsend{}

\pgfsetarrowsend{latex} 
	\draw[sky, very thick] (-.5, -1.5) -- (1.5, -1.5);
	\draw[sky, very thick] (-.5, .5) -- (1.5, .5);
	\draw[sky, very thick] (-.5, -1.5) -- (-.5, .5);
	\draw[sky, very thick] (1.5, -1.5) -- (1.5, .5);
\pgfsetarrowsend{} 
\end{tikzpicture}}
\caption{The zig zag framework has a gain graph with two vertices (a). Realized as a framework on the $2$-dimensional torus (b). The derived framework is shown in (c). A non-generic position of the vertices on $\Tor^2$ (d). The framework pictured in (d) is not infinitesimally rigid, but the framework $\pofw$ shown in (b) is infinitesimally rigid on $\Tor^2$,  and the corresponding derived framework $\dpfwo$ (c) is infinitesimally rigid in $\mathbb R^2$.  \label{fig:zigzag} \index{zig-zag framework}}
\end{center}
\end{figure}

As stated, a framework on $\Tor^d$ is infinitesimally rigid if and only if the only infinitesimal motions of the framework are infinitesimal translations. In addition, any periodic framework $\pofw$ on $\Tor^d$ has a $d$-dimensional space of trivial motions. It follows that the rigidity matrix always has at least $d$ trivial solutions, and hence
\begin{thm}
A periodic orbit framework $\pofw$ is infinitesimally rigid on the fixed torus $\Tor^d$ if and only if the rigidity matrix $\R_0\pofw$ has rank $d|V|-d$. 
\label{thm:fixedMatrixRank}
\end{thm}
The rigidity matrix of the framework in Example \ref{ex:rigMatrix} above has rank 6, which is exactly $2|V|-2$, and hence $\pofw$ is infinitesimally rigid on $\Tor^2$. 

\begin{ex}[the {\it zig-zag} framework, Figure \ref{fig:zigzag}] \index{zig-zag framework}
Consider the graph $G= (V, E)$ where $V = \{v_1, v_2\}$ and $E$ consists of two copies of the edge connecting the two vertices $v_1$ and $v_2$ (Figure \ref{fig:zigzag}a). If the gains on the two edges are the same, then the framework is not infinitesimally rigid, since both rows of the rigidity matrix will be identical. Let $\bm$ be a gain assignment on $G$ with $\bm_1 \neq \bm_2$.  The periodic orbit framework $\pofw$ is infinitesimally rigid on $\Tor^2$ if and only if:
\begin{enumerate}
	\item $\p_1 \neq \p_2$ 
	\item both edges have distinct directions (that is, the vectors $\p_1 - \p_2 - \bm_1$ and $\p_1 - \p_2 - \bm_2$ are independent). See Figure \ref{fig:zigzag}d. 
\end{enumerate}
Figures \ref{fig:zigzag}b and \ref{fig:zigzag}c depict $\pofw$ on  $\Tor^2$ and $\dpfwo$ in $\mathbb R^2$ respectively.
\qed
\end{ex}

There are a number of simple observations which we record here for future reference.  
\begin{cor} 
A periodic orbit framework $\pofw$ where $G$ has $|E| < d|V|-d$ is not infinitesimally rigid on $\Tor^d$. 
\end{cor}

A collection of edges $E' \subset E$ of the periodic orbit framework $\pofw$ is called {\it independent} \index{periodic orbit framework!independent on $\Tor^d$} if the corresponding rows of the rigidity matrix are linearly independent. For each set of multiple edges $e_{i_1} = e_{i_2} = \cdots = e_{i_t}$ in $E$, we can have at most $d$ independent copies. If a framework $\pofw$ has edges corresponding to dependent rows in the rigidity matrix, we say that the edges are {\it dependent}. We may also refer to a framework $\pofw$ as being independent or dependent, and for clarity we will at times write {\it dependent on $\Tor^d$} to differentiate this setting from the finite case (frameworks which are not on a torus). 

\begin{cor}
Any periodic orbit framework $\pofw$ where $G$ has $|E| > d|V|-d$ is dependent on $\Tor^d$.
\label{cor:tooManyEdges}
\end{cor}
We sometimes call such a framework {\it over-counted}. \index{over-counted} A periodic orbit framework $\pofw$ whose underlying gain graph satisfies $|E| = d|V|-d$ and is infinitesimally rigid on $\Tor^d$ will be called {\it minimally rigid}. \index{minimally rigid} \index{periodic orbit framework!minimally rigid} In other words, a minimally  rigid framework on $\Tor^d$ is one that is both infinitesimally rigid and independent. In fact, such a framework is maximally independent -- adding any new edge will introduce a dependence among the edges. If a periodic orbit framework is minimally rigid, then the removal of any edge will result in a framework that is not infinitesimally rigid.

We observe a periodic analogue of the extension of Maxwell's rule, which provides simple combinatorial necessary conditions for rigidity.
\begin{cor}
Let $\pofw$ be a minimally rigid periodic orbit framework. Then 
\begin{enumerate}
	\item $|E| = d|V| - d$, and 
	\item for all subgraphs $G' \subseteq G$, $|E'| \leq d|V'| - d$. 
\end{enumerate}
\label{cor:perMaxwell}
\end{cor}

\begin{cor}
Any loop edge in the $d$-periodic orbit framework $\pofw$ is dependent on $\Tor^d$. 
\label{cor:loopsAreDependent}
\end{cor}
The following useful result is a direct consequence of the fact that the row rank of a matrix is equal to its column rank. \index{periodic orbit framework!loops}		
\begin{cor}		
A $d$-periodic framework $\pofw$ whose underlying gain graph satisfies $|E| = d|V|-d$ is independent on $\Tor^d$ if and only if it is infinitesimally rigid on $\Tor^d$. Moreover, the vector space of non-trivial infinitesimal motions of $\pofw$ is isomorphic to the vector space of row dependencies of $\R_0\pofw$. 
\label{cor:rowRankColumnRank}
\end{cor}

We also now confirm that if $\pofw$ is infinitesimally rigid, then all frameworks that are $\Tor^d$-congruent to $\pofw$ are also infinitesimally rigid. The proof is a straightforward application of the definition of $\Tor^d$-congruence.
\begin{prop}
Let $\pofw$ and $(\pogn, q)$ be $\Tor^d$-congruent. Then 
\[\rank \R_0\pofw = \rank \R_0(\pogn, q).\]
\end{prop}

The rows of $\R_0\pofw$ corresponding to edges with zero gains can be viewed as rows in the rigidity matrix of a finite framework, as described in any introduction to rigidity; see \cite{CountingFrameworks} or \cite{SomeMatroids}, for example. Since at most $d|V| - {d+1 \choose 2}$ rows can be independent in the finite matrix, we have the following proposition:
\begin{prop}
Let $\pog$ be a $d$-periodic orbit graph with all edges having zero gains,  $\bm = 0$. If $|E| > d|V| - {d+1 \choose 2}$, then the edges of $\pofw$ are dependent for any configuration $\p$. 
\label{prop:zeroGains}
\end{prop} 

Because loop edges are always dependent by Corollary \ref{cor:loopsAreDependent}, we restrict our attention to frameworks $\pofw$ that do not have loop edges. On the flexible torus, however, loops may be independent, but we do not consider that case here.

\begin{rem} The derived periodic framework corresponding to the periodic orbit framework in Example \ref{ex:4vertex} would not be considered minimally rigid as an infinite framework in the sense of being both independent and rigid. That is, disregarding the periodic qualities of the graph and recording an infinite dimensional rigidity matrix, it is not true that row rank equals column rank, and hence Corollary \ref{cor:rowRankColumnRank} is no longer true. Further details on this problem can be found in Guest and Hutchinson, \cite{DeterminancyRepetitive}.
\qed \end{rem}

\begin{rem} We can define a {\it $d$-periodic rigidity matroid} $\mathcal R_0\pofw$ on the edges of the $d$-periodic orbit framework: A set of edges is independent in the rigidity matroid $\mathcal R_0\pofw$ if the corresponding rows are independent in the periodic rigidity matrix $\R_0\pofw$. \index{matroid, periodic}
\qed \end{rem}

\begin{rem}
We can also use the structure we have developed to define a periodic rigidity matrix $\R$ for frameworks on the flexible torus. In this case, the rigidity matrix has dimension $|E| \times d|V| + {d+1 \choose 2}$, with one additional column for each variable entry in $L(t)$. Recall that there are ${d+1 \choose 2}$ non-zero entries in a $d \times d$ lower triangular matrix. Since the only trivial motions are the infinitesimal translations (that is, we don't get any new trivial motions on the flexible torus), we obtain the following flexible torus version of Theorem \ref{thm:fixedMatrixRank}:
\begin{thm}
A periodic orbit framework $\pofw$ is infinitesimally rigid on the flexible torus $\T^d$ if and only if the rigidity matrix $\R\pofw$ has rank $d|V|+ {d \choose 2}$. 
\end{thm}
This result is confirmed in the work of Borcea and Streinu \cite{periodicFrameworksAndFlexibility}. Furthermore it is possible to consider intermediate cases, where only some $k < {d+1 \choose 2}$ of the entries of the lattice matrix $L(t)$ are variable. The consideration of this case is left to \cite{myThesis}.
\end{rem}

\subsection{Stresses and independence}

A row dependence among the rows of the rigidity matrix can be thought of as a {\it stress} on the edges of the periodic orbit matrix, or equivalently a {\it periodic stress}  on the edges of a periodic framework. This topic has been considered by Guest and Hutchinson \cite{DeterminancyRepetitive}. The minimally rigid graphs are therefore the graphs that do not have any infinitesimal motions, or any stresses among their edges. In finite rigidity, this state is called {\it isostatic}, \index{isostatic} but we avoid this terminology here for the reasons outlined in \cite{DeterminancyRepetitive}. Borcea and Streinu also define stresses for $d$-periodic frameworks with a flexible lattice in \cite{periodicFrameworksAndFlexibility}.

\subsection{The unit torus and affine transformations }
\label{sec:theUnitTorus}
In this section we show that frameworks on the unit torus can be used to model all $d$-periodic orbit frameworks, since Theorem \ref{thm:affineInvariance} will demonstrate that infinitesimal rigidity of periodic orbit frameworks is affinely invariant. 

An {\it affine transformation} is a map $A:  \mathbb{R}^d  \longrightarrow  \mathbb{R}^d$, with  $x  \longmapsto   x B + t$,
where $B$ is an invertible $d \times d$ matrix, and $t \in \mathbb{R}^d$. The next result concerning the affine invariance of independence on $\Tor^d$ was shown independently in \cite{periodicFrameworksAndFlexibility}, and we omit the proof, which is straightforward. \index{affine invariance!fixed torus}
\begin{thm}
Let $\pofw$ be a $d$-periodic orbit framework on $\Tor^d$. Let $L_0$ be the $d \times d$ lattice matrix whose rows are the generators of $\Tor^d$. Let $A$ be an affine transformation of $\mathbb{R}^d$, with $A(x) = xB + t$, and where $A(p) = (A(p_1), \dots, A(p_{|V|})) \in \mathbb R^{d|V|}$. Then the edges of $\pofwA$ are independent on $\mathbb{R}^d/L_0B\mathbb Z^d$ if and only if the edges of $\pofw$ are independent on $\Tor^d = \mathbb{R}^d/L_0 \mathbb Z^d$. 
\label{thm:affineInvariance}
\end{thm}

\begin{cor}
Let $\mathcal F = \pofw$ be a $d$-periodic orbit framework on $\Tor^d$, where $L_0$ is the $d \times d$ matrix of generators of $\Tor^d$. Let $\mathcal F'$ be the image of $\mathcal F$ under the unique affine transformation of $\mathbb R^d$ which maps $L_0$ to the $d$-dimensional identity matrix $I_{d \times d}$. Then $\mathcal F$ is infinitesimally rigid on $\Tor^d$ if and only if $\mathcal F'$ is infinitesimally rigid on the $d$-dimensional unit torus, $\Unit^d$. 
\label{cor:unitTorus}
\end{cor}

\begin{rem}
It is essential that the affine transformation of Corollary \ref{cor:unitTorus} act on both the points of the framework, and the generators of the torus (the rows of $L_0$). In other words, it is {\it not } true that a framework $\pofw$ is infinitesimally rigid on $\Tor^d$ if and only if an affine image of the framework $\pofw$ is infinitesimally rigid on $\Tor^d$. The framework pictured in Figure \ref{fig:affineCounter} is an example. \qed \end{rem}

\begin{verse}
\begin{figure}
\begin{center}
\subfloat[]{\begin{tikzpicture} 
\tikzstyle{vertex1}=[circle, draw, fill=couch, inner sep=1pt, minimum width=4pt]; 
\tikzstyle{vertex2}=[circle, draw, fill=lips, inner sep=1pt, minimum width=4pt]; 
\tikzstyle{vertex3}=[circle, draw, fill=melon, inner sep=1pt, minimum width=4pt]; 
\tikzstyle{vertex4}=[circle, draw, fill=bluey, inner sep=1pt, minimum width=4pt]; 

\node[vertex1] (1) at (1,-1) {};
\node[vertex1] (2) at (0,0) {};

\draw[thick] (1) -- (2)  ;
\draw[thick] (1) -- (1.5, -0.5);
\draw[thick] (-0.5, -0.5) -- (2);

\pgfsetarrowsend{latex} 
	\draw[sky, very thick] (-.5, -1.5) -- (1.5, -1.5);
	\draw[sky, very thick] (-.5, .5) -- (1.5, .5);
	\draw[sky, very thick] (-.5, -1.5) -- (-.5, .5);
	\draw[sky, very thick] (1.5, -1.5) -- (1.5, .5);
\pgfsetarrowsend{} 
\end{tikzpicture}}\hspace{1cm}
\subfloat[]{\begin{tikzpicture} 
\tikzstyle{vertex1}=[circle, draw, fill=couch, inner sep=1pt, minimum width=4pt]; 
\tikzstyle{vertex2}=[circle, draw, fill=lips, inner sep=1pt, minimum width=4pt]; 
\tikzstyle{vertex3}=[circle, draw, fill=melon, inner sep=1pt, minimum width=4pt]; 
\tikzstyle{vertex4}=[circle, draw, fill=bluey, inner sep=1pt, minimum width=4pt]; 

\pgfsetarrowsend{latex} 
	\draw[sky, very thick] (-.5, -1.5) -- (1.5, -1.5);
	\draw[sky, very thick] (-.5, .5) -- (1.5, .5);
	\draw[sky, very thick] (-.5, -1.5) -- (-.5, .5);
	\draw[sky, very thick] (1.5, -1.5) -- (1.5, .5);
\pgfsetarrowsend{} 

\clip (-.5, -1.5) --  (1.5, -1.5) -- (1.5, .5) -- (-.5, .5) --(-.5, -1.5);

\pgftransformyshift{-.3cm}
\node[vertex1] (1) at (1,-1) {};
\node[vertex1] (2) at (0,0) {};

\draw[thick] (1) -- (2)  ;
\draw[thick] (1) -- (1.5, -0.5);
\draw[thick] (-0.5, -0.5) -- (2);

\pgfsetarrowsend{stealth}
\draw[pink, very thick] (1) -- (1, -.6);
\draw[pink, very thick] (2) -- (0, .4);
\pgfsetarrowsend{}

\pgftransformyshift{.3cm}

\pgftransformyshift{.3cm}
\node[vertex1] (1) at (1,-1) {};
\node[vertex1] (2) at (0,0) {};

\draw[thick] (1) -- (2)  ;
\draw[thick] (1) -- (1.5, -0.5);
\draw[thick] (-0.5, -0.5) -- (2);
\pgftransformyshift{-.3cm}

\pgftransformyshift{.9cm}
\node[vertex1] (1) at (1,-1) {};
\node[vertex1] (2) at (0,0) {};

\draw[thick] (1) -- (2)  ;
\draw[thick] (1) -- (1.5, -0.5);
\draw[thick] (-0.5, -0.5) -- (2);
\pgftransformyshift{-.9cm}

\pgftransformyshift{-.9cm}
\node[vertex1] (1) at (1,-1) {};
\node[vertex1] (2) at (0,0) {};

\draw[thick] (1) -- (2)  ;
\draw[thick] (1) -- (1.5, -0.5);
\draw[thick] (-0.5, -0.5) -- (2);
\pgftransformyshift{.9cm}

\pgftransformyshift{1.5cm}
\node[vertex1] (1) at (1,-1) {};
\node[vertex1] (2) at (0,0) {};

\draw[thick] (1) -- (2)  ;
\draw[thick] (1) -- (1.5, -0.5);
\draw[thick] (-0.5, -0.5) -- (2);
\pgftransformyshift{-1.5cm}

\pgftransformyshift{-1.5cm}
\node[vertex1] (1) at (1,-1) {};
\node[vertex1] (2) at (0,0) {};

\draw[thick] (1) -- (2)  ;
\draw[thick] (1) -- (1.5, -0.5);
\draw[thick] (-0.5, -0.5) -- (2);
\pgftransformyshift{1.5cm}

\end{tikzpicture}}
\caption{The framework pictured in (a) is infinitesimally rigid on the fixed torus $\Tor^2$. The affine transformation of the framework shown in (b), without a corresponding affine transformation of $\Tor^2$,  is {\it not} infinitesimally rigid, as indicated.  \label{fig:affineCounter}}
\end{center}
\end{figure}
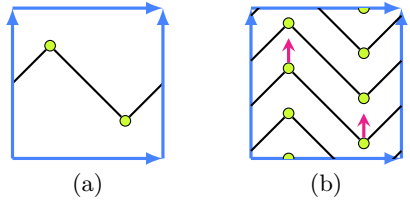\end{verse}

\subsection{Generic frameworks}
\label{sec:genericFrameworks}
\index{periodic orbit framework!generic|(}
Let $V$ be a finite set of vertices, and let $\p$ be a realization of these vertices on to the $d$-dimensional unit torus $\Unit^d = [0,1)^d$.  Let $k \in \mathbb Z_+$ be given, and let $K$ be the set of all edges between pairs of vertices of $V$ with gains $\bm_e = (m_{e,1}, m_{e,2}, \dots, m_{e,d})$ where $|m_{e,i}| \leq k$ for $i=1, \dots d$. Then $K$ is the set of all edges with bounded gains.

Consider a set of edges $E \subset K$ such that, for some realization $\p$, the rows of $\R_0$ corresponding to $E$ are independent. The determinants of the $|E| \times |E|$ submatrices of these rows will either be identically zero or will define an algebraic variety in $\mathbb R^{d|V|}$ (by setting these determinants equal to zero, and taking the $\p_i$'s as variables). The collection of all such varieties, corresponding to all such subsets $E$ will define a closed set of measure zero (this set is a finite union of closed sets of measure zero). Let this set be denoted $\mathcal X_k$. The complement of $\mathcal X_k$ in $\mathbb R^{d|V|}$ is an open dense set in $\mathbb R^{d|V|}$, and hence its restriction to the subspace of realizations $\p$ of the vertices $V$ on the unit torus, $[0,1)^{d|V|}$ is also open and dense.

Any realization $\p$ of the vertex set $V$ where $\p \notin \mathcal X_k$ will be called {\it k-generic} (recall that $k$ was the upper bound on the gain assignments).  More generally, we may consider graphs that are $k$-generic for any $k$. By the Baire Category Theorem, the countable intersection 
\[\bigcap_{k \in \mathbb Z} \big( \mathbb R^{d|V|} - \mathcal X_k \big)\]
is dense in $\mathbb R^{d|V|}$, as the intersection of open dense sets in the Baire space $\mathbb R^{d|V|}$ \cite{munkres}. We refer to a realization in this set as simply {\it generic}, and it is this definition that we use throughout the remainder of this paper. 

%
%
%
%
%
\begin{cor} {\rm  (to Theorem \ref{thm:affineInvariance})}
Let $A$ be an affine transformation of $\mathbb R^d$ which maps $L_0$ to the identity matrix $I_{d \times d}$, and let $\pog$ be a periodic orbit graph. $A(\p)$ is a generic realization of $\pog$ on the unit torus, $\Unit^d = [0,1)^d$ if and only if $\p$ is a generic realization of $\pog$ on $\Tor^d$.
\end{cor}

As a consequence of this result, from this point forward we assume that all frameworks are realized on the unit torus.  That is, $\p: V \longrightarrow [0,1)^d$, and $L_0 = I_{d \times d}$, the identity matrix. We continue to write $\Tor^d$, but drop the matrix ``$L_0$" from expressions involving gains, since $m_eL_0 = m_e I_{d \times d} = m_e$.

The following result states that for a given $d$-periodic orbit graph, all generic realizations share the same rigidity properties. Compare Lemma 2.2.1 in \cite{SomeMatroids}.

\begin{lem}[\bf Special Position Lemma] 
Let $\pog$ be a $d$-periodic orbit graph, and suppose that for some realization $\p_0$ of $\pog$ on $\Tor^d$ the framework $(\pog, \p_0)$ is infinitesimally rigid. Then for all generic realizations $\p$ of $\pog$ on $\Tor^d$, the framework $\pofw$ is infinitesimally rigid.
\label{lem:specialPosition}
\end{lem}
\index{special position lemma}

\begin{proof}
Since the framework $(\pog, \p_0)$ is infinitesimally rigid on $\Tor^d$, the rigidity matrix for  $(\pog, \p_0)$ has maximum rank, $\rank \R_0(\pog, p_0) = d|V|-d$ (Theorem \ref{thm:fixedMatrixRank}). By definition of generic, any framework $\pofw$ with $\p$ generic will have 
$$ \rank \R_0 \pofw \geq \rank  \R_0(\pog, \p_0)$$
It follows that $\rank  \R_0\pofw = d|V|-d$, and the framework $\pofw$ is infinitesimally rigid.  %
\end{proof}

The following modification of the Special Position Lemma states that the coordinates of $\p_0$ need not be on the unit torus, but can in fact be taken anywhere in $\mathbb R^{d|V|}$. 
\begin{cor}[\bf Modified Special Position Lemma]
Let $\pog$ be a $d$-periodic orbit graph, and suppose that for some realization $\p_0: V \rightarrow \mathbb R^{d|V|}$ the rigidity matrix $\R_0(\pog, \p_0)$ has rank $d|V|-d$.  Then for all generic realizations $\p$ of $\pog$ on $\Tor^d = [0,1)^d$, the framework $\pofw$ is infinitesimally rigid.
\label{lem:specialPositionMod}
\end{cor}
\index{special position lemma!modified}

\begin{proof}
Recall that the set of generic realizations of a vertex set $V$ is dense in $\mathbb R^{d|V|}$, and that the set of generic realizations on the torus is simply the restriction of this larger set to $[0, 1)^{d|V|}$. By the arguments of the proof of Lemma \ref{lem:specialPosition}, if $\R_0(\pog, \p_0) = d|V|-d$ for {\it some} realization $p_0 \in \mathbb R^{d|V|}$, then $\R_0(\pog, \p) = d|V|-d$ for all generic realizations in $\mathbb R^{d|V|}$, which includes all generic realizations on $[0, 1)^{d|V|}$. %
\end{proof}

This result should be understood to mean that we can pick any representatives of a vertex, provided that the edge representatives are the same, in the sense that the corresponding rows of the rigidity matrix are unchanged.  In light of these results, we may say that a periodic orbit graph $\pog$ is {\it generically rigid} on $\Tor^2$, meaning that the periodic orbit framework $\pofw$ is rigid for all generic realizations $p$ of the vertices of $G$.
\index{periodic orbit framework!generic|)}

\subsection{Infinitesimal rigidity implies rigidity}
\label{sec:averaging}

It is possible to prove a periodic version of the Asimow and Roth result which demonstrates that for generic frameworks, infinitesimal rigidity and rigidity are equivalent \cite{AsimowRoth}. The central ideas of their proof carry over to the periodic context, since we are working with a finite matrix corresponding to a finite graph on a torus. The full development of this idea can be found in \cite{myThesis}.

It is more straightforward to show that infinitesimal rigidity always implies rigidity for periodic orbit frameworks. In \cite{myThesis} we use the averaging technique to do this (see also \cite{UnpublishedRigidityBookChapter2}); however there are a number of proofs that could easily be adapted to the periodic setting, and we do not include the details here. 

\subsection{$T$-gain procedure preserves infinitesimal rigidity on $\Tor^d$}
\label{sec:TGainsPreserveInfinitesimalRigidity}

\index{T-gain procedure!on $\Tor^d$}
In section \ref{sec:Tgains} we described the $T$-gain procedure for identifying the local gain group of a graph.  We noted that the original gain assignment $\bm$ and the $T$-gain assignment $\bm_T$ can be seen as simply two different ways to describe the same infinite periodic graph. Most importantly, we now confirm that the rigidity matrices corresponding to these two periodic orbit graphs have the same rank. In fact this is a {\it geometric} statement, with a generic corollary. 

\begin{thm}
For any framework $\pofw$, 
$$\rank \R_0\pofw = \rank \R_0(\tgT, p'), $$
where $\p':V \rightarrow \mathbb R^d$ is given by $\p'_i = \p_i + \bm_T(v_i)$.
\label{thm:TgainsPreserveRigidity}
\end{thm}
The essence of the following argument is that the $T$-gain procedure changes the representatives of the vertices used in the rigidity matrix, which, together with the new gains, leaves the rows of the matrix unchanged. 

\begin{proof}
Let $T$ be a spanning tree in $\pog$. Each vertex $v_i$ of $G$ is labeled with a $T$-potential, which we denote $\bm(v_i, T) = \bm_T(i)$. The edge $e = \{v_i, v_j; \bm_e\}$ has $T$-gain 
\[\bm_T(e) = \bm_T(i) + \bm_e - \bm_T(j).\]
We know that the derived graphs $G^{\bm}$ and $G^{\bm_T}$ are isomorphic by Theorem \ref{thm:TGainIsomorphic}. For each vertex $v \in V$, we relabel the indices of the vertices in the fibre over $v$ according to the rule 
\[z \longrightarrow z - \bm_T(v).\]
In other words, the vertex $(v_{i}, z)$ in $G^{\bm}$, where $z \in \mathbb Z^d$ is mapped to the vertex $(v_{i}, z - \bm_T(i))$ in $G^{\bm_T}$.

Suppose that a set of rows is dependent in $\R_0\pofw$. Then there exists a vector of scalars, say $\omega = [\begin{array}{ccc} \omega_1 & \cdots & \omega_{|E|} \end{array}]$ such that 
\[\omega \cdot \R_0\pofw = 0.\]
As in the proof of affine invariance, for a particular vertex we consider the edges directed into and out from the vertex separately. That is, for a vertex $v_i \in V$,  let $E_+$ denote the set of edges directed {\it out} from the vertex $v_i$, and let $E_-$ denote the set of edges directed {\it into} the vertex $v_i$.  For each vertex $v_i \in V$ the column sum of $\R_0\pofw$ becomes 
\begin{equation}
\sum_{ e_{\alpha} \in E_+} \omega_{e_{\alpha}} (p_i - (p_j+m_{e_{\alpha}})) + 
\sum_{ e_{\beta} \in E_-} \omega_{e_{\beta}} (p_i - (p_k-m_{e_{\beta}})) = 0.
\label{eqn:tGain}
\end{equation}

Adding and subtracting $\bm_{T}(i)$ and $\bm_T(j)$ to the first summand of (\ref{eqn:tGain}), we obtain 
\[\sum_{ e_{\alpha} \in E_+} \omega_e \Big(\p_i - \p_j + \bm_T(i) - \bm_T(j) -[\bm_T(i) + \bm_e - \bm_T(j)]\Big),\]
which is equivalent to
\[\sum_{ e_{\alpha} \in E_+} \omega_e \Big(\p_i+ \bm_T(i)  - (\p_j + \bm_T(j)) -\bm_T(e)\Big).\]
Similarly, the second summand of (\ref{eqn:tGain}) becomes 
\[\sum_{ e_{\beta} \in E_-} \omega_e \Big(\p_i+ \bm_T(i)  - (\p_j + \bm_T(j)) +\bm_T(e)\Big).\]
Putting them together, (\ref{eqn:tGain}) becomes the column sum of the column of  $\R_0(\tgT, p')$ corresponding to the vertex $v_i$.  Hence this set of rows is dependent in $\R_0\pofwT$. The argument reverses for the converse. %
\end{proof}

This geometric result has the following generic corollary, which implies that if $\pofw$ is infinitesimally rigid for generic $p$, then $(\tgT, p)$ is also infinitesimally rigid for the same position $p$. The graphs pictured in Figure \ref{fig:TgainHomotopy} are an example. 

\begin{cor}
The periodic orbit graph $\pog$ is generically rigid on $\Tor^d$ if and only if $\tgT$ is generically rigid on $\Tor^d$. 
\label{cor:TgainsPreserveGeneric}
\end{cor}
\begin{proof}
Let $p$ be a generic position of $\pog$ on $\Tor^d$. Let $\p'_i = \p_i + \bm_T(v_i)$. While $\p: V \longrightarrow \Unit^d$, $\p': V \longrightarrow \mathbb R^{d|V|}$. 
By the Modified Special Position Lemma (\ref{lem:specialPositionMod}), the rank of the matrix $\R_0\pofwT$ is generically the same as the rank of the matrix $\R_0 (\langle G, \bm_T \rangle, \p')$, which, by Theorem \ref{thm:TgainsPreserveRigidity}, is the same as the rank of the matrix $\R_0\pofw$. %
\end{proof}


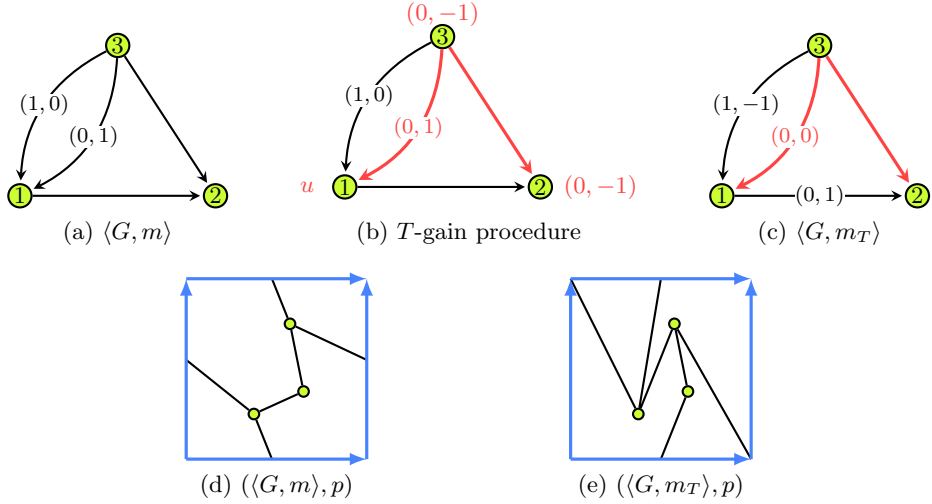
\begin{figure}\begin{center}
\subfloat[$\pog$]{\begin{tikzpicture}[->,>=stealth,shorten >=1pt,auto,node distance=2.8cm,thick, font=\footnotesize] 
\tikzstyle{vertex1}=[circle, draw, fill=couch, inner sep=.5pt, minimum width=3.5pt, font=\footnotesize]; 
\tikzstyle{vertex2}=[circle, draw, fill=melon, inner sep=.5pt, minimum width=3.5pt, font=\footnotesize]; 
\tikzstyle{voltage} = [fill=white, inner sep = 0pt,  font=\scriptsize, anchor=center];

	\node[vertex1] (1) at (-1.3,0)  {$1$};
	\node[vertex1] (2) at (1.3,0) {$2$};
	\node[vertex1] (3) at (0, 2) {$3$};
		\draw[thick] (1) --  (2);
		\draw[thick] (3) --  (2);
	\draw[thick] (3) edge  [bend right]  node[voltage] {$(1,0)$} (1);
	\draw[thick] (3) edge  [bend left] node[voltage] {$(0,1)$} (1);
\end{tikzpicture}}\hspace{.7cm}
\subfloat[$T$-gain procedure]{\begin{tikzpicture}[->,>=stealth,shorten >=1pt,auto,node distance=2.8cm,thick, font=\footnotesize] 
\tikzstyle{vertex1}=[circle, draw, fill=couch, inner sep=.5pt, minimum width=3.5pt, font=\footnotesize]; 
\tikzstyle{vertex2}=[circle, draw, fill=melon, inner sep=.5pt, minimum width=3.5pt, font=\footnotesize]; 
\tikzstyle{voltage} = [fill=white, inner sep = 0pt,  font=\scriptsize, anchor=center];	
	
		\node[vertex1] (1) at (-1.3,0)  {$1$};
	\node[vertex1] (2) at (1.3,0) {$2$};
	\node[vertex1] (3) at (0, 2) {$3$};
		\draw[thick] (1) --  (2);
		\draw[very thick, melon] (3) --  (2);
	\draw[thick] (3) edge  [bend right]  node[voltage] {$(1,0)$} (1);
	\draw[very thick, melon] (3) edge  [bend left] node[voltage] {$(0,1)$} (1);
	
	\node[melon] at (0, 2.3) {$(0,-1)$};
	\node[melon] at (-1.8, 0) {$u$};
	\node[melon] at (2.1, 0) {$(0, -1)$};
	
\end{tikzpicture}}\hspace{.7cm}
\subfloat[$\tgT$]{\begin{tikzpicture}[->,>=stealth,shorten >=1pt,auto,node distance=2.8cm,thick, font=\footnotesize] 
\tikzstyle{vertex1}=[circle, draw, fill=couch, inner sep=.5pt, minimum width=3.5pt, font=\footnotesize]; 
\tikzstyle{vertex2}=[circle, draw, fill=melon, inner sep=.5pt, minimum width=3.5pt, font=\footnotesize]; 
\tikzstyle{voltage} = [fill=white, inner sep = 0pt,  font=\scriptsize, anchor=center];

\node[vertex1] (1) at (-1.3,0)  {$1$};
	\node[vertex1] (2) at (1.3,0) {$2$};
	\node[vertex1] (3) at (0, 2) {$3$};
		\draw[thick] (1) -- node[voltage] {$(0,1)$} (2);
		\draw[very thick, melon] (3) --  (2);
	\draw[thick] (3) edge  [bend right]  node[voltage] {$(1,-1)$} (1);
	\draw[very thick, melon] (3) edge  [bend left] node[voltage] {$(0,0)$} (1);
\end{tikzpicture}}\\

\subfloat[$\pofw$]{\begin{tikzpicture}[thick,scale=1.2] 
\tikzstyle{vertex1}=[circle, draw, fill=couch, inner sep=0pt, minimum width=4pt]; 
\tikzstyle{vertex2}=[circle, draw, fill=mango, inner sep=0pt, minimum width=4pt]; 

\node[vertex1] (1) at (-.25,-.5) {};
\node[vertex1] (2) at (.3,-.25) {};
\node[vertex1] (3) at (.15,.5) {};

\draw[thick] (1) -- (2) -- (3) (3) -- (1,0.1) (-1,0.1) -- (1) (3) -- (-.05,1) (-.05,-1)--(1);
\pgfsetarrowsend{latex} 
	\draw[sky, very thick] (-1, -1) -- (1, -1); 
	\draw[sky, very thick] (-1, 1) -- (1, 1);
	\draw[sky, very thick] (-1, -1) -- (-1, 1);
	\draw[sky, very thick] (1, -1) -- (1, 1);
\pgfsetarrowsend{} 
\end{tikzpicture}}
\hspace{1in}
\subfloat[$(\tgT,p)$]{\begin{tikzpicture}[thick,scale=1.2] 
\tikzstyle{vertex1}=[circle, draw, fill=couch, inner sep=0pt, minimum width=4pt]; 
\tikzstyle{vertex2}=[circle, draw, fill=mango, inner sep=0pt, minimum width=4pt]; 

\node[vertex1] (1) at (-.25,-.5) {};
\node[vertex1] (2) at (.3,-.25) {};
\node[vertex1] (3) at (.15,.5) {};

\draw[thick] (1) -- (3) -- (2) (1) -- (0,1) (0,-1) -- (2) (3) -- (1, -1) (-1,1) -- (1);
\pgfsetarrowsend{latex} 
	\draw[sky, very thick] (-1, -1) -- (1, -1); 
	\draw[sky, very thick] (-1, 1) -- (1, 1);
	\draw[sky, very thick] (-1, -1) -- (-1, 1);
	\draw[sky, very thick] (1, -1) -- (1, 1);
\pgfsetarrowsend{} 
\end{tikzpicture}} \end{center}
\caption{The $T$-gain procedure used on a graph $\pog$ to form $\tgT$ is shown in (a) -- (c). The periodic orbit frameworks on $\Tor^2$ are shown in (d) and (e), corresponding to the periodic orbit graphs in (a) and (c) respectively. \label{fig:TgainHomotopy}}
\end{figure}

\subsection{Gain assignments and infinitesimal rigidity}
\label{sec:aKnownResult}
The following theorem says that given a graph $G$ with certain combinatorial properties, we can always find an appropriate gain assignment $\bm$ and geometric realization $\p$ to yield a minimally rigid framework on $\Tor^d$.

\begin{thm}[Whiteley, \cite{UnionMatroids}]
For a multigraph $G$, the following are equivalent:
\begin{enumerate}[(i)]
\item $G$ satisfies $|E| = d|V|-d$, and every subgraph $G' \subseteq G$ satisfies $|E'| \leq d|V'| - d$, 
\item $G$ is the union of $d$ edge-disjoint spanning trees,
\item For some gain assignment $m$ and some realization $p$, the framework $\pofw$  is minimally  rigid on $\Tor^d$.
\end{enumerate}
\label{waltersThm}
\end{thm}

%
%

The proof from \cite{UnionMatroids}  constructs some gain assignments which are sufficient for infinitesimal rigidity (in fact, it produces an infinite space of such gains).  In a nutshell it says that given any graph satisfying the necessary conditions of Corollary \ref{cor:perMaxwell}, we can define a gain assignment, with basis vector gains, that will be infinitesimally rigid on $\Tor^2$. It is true, however, that these are not the only infinitesimally rigid frameworks. The question of interest then becomes: 
\begin{question}
When is a periodic orbit graph $\pog$ generically rigid on $\Tor^d$?
\end{question}
The goal of a subsequent paper will be devoted to broadening the scope of Theorem \ref{waltersThm} for periodic orbit frameworks on the two-dimensional fixed torus, and to characterize more precisely the interactions between combinatorics, geometry and topology in defining rigid frameworks. In a recent paper \cite{BorceaStreinuII} the authors offer an improvement of \ref{waltersThm} for the flexible torus. 

As previously noted, the approach of Borcea and Streinu \cite{periodicFrameworksAndFlexibility} {\it does not} consider the gains to be part of the combinatorial information of a periodic framework. Instead they work with the notion of generic edge directions, which involves both the gain and the position of the vertices.  We will consider the gains of a periodic orbit framework to be part of the combinatorial information of the graph, and will characterize the rigidity of periodic orbit frameworks for {\it all} gains. 

Malestein and Theran \cite{Theran} do consider gain graphs. In their language, our gain graphs are ``coloured graphs". We now turn to the task of identifying necessary conditions on the gain assignments for infinitesimal rigidity on the fixed torus. 

\section{Necessary conditions on gains for rigidity}
\label{sec:necessary}
\subsection{Necessary conditions for infinitesimal rigidity on $\Tor^d$}
In this section we establish necessary conditions on the gains of a periodic orbit graph $\pog$ for it to be infinitesimally rigid on $\Tor^d$. Here is a preliminary necessary condition for infinitesimal rigidity on $\Tor^d$. Recall that for a gain graph $ \pog$ with cycle space $\C(G)$, the gain space $\gs(G)$ is the vector space (over $\mathbb Z$) spanned by the net gains on the cycles of $\C(G)$. 

\begin{thm} Let $\pog$ be a $d$-periodic orbit graph with $|E| = d|V|-d$. If $\pofw$ is infinitesimally rigid for some realization $\p$, then every subgraph $G' \subseteq G$ with $|E'| = d|V'| - d$ has $|\gs(G')| \geq d-1$. 
\label{thm:nDimConstructive}
\end{thm}
 
 \begin{proof}
Suppose $G' \subseteq G$ has $|\gs(G')| = k$, where $k < d-1$. Performing the $T$-gain procedure if necessary, the gains of the edges of $G'$ are zero on at least two coordinates, say $x$ and $y$. The basic idea of this proof is that such a framework is disconnected in the $xy$-plane, and we can apply a rotation in this plane. Suppose without loss of generality that all edges of $\pog$ have gains $\bm_e = (0,0,m_{e3}, \dots, m_{ed}) \in \mathbb Z^d$. Let $\p = (p_1, p_2, \dots, p_d)$ be a point in $\Tor^d$. Let $v = (-p_2, p_1, 0, \dots, 0)$. Then 
$$v\cdot (\p_i - (\p_j + \bm_eL_0)) = (-p_2, p_1, 0, \dots, 0)\cdot (\p_{i1} - \p_{j1}, \p_{i2} - \p_{j2}, \dots )$$ 
which is a rotation in the plane of the first two coordinates, of a finite (i.e. not periodic) framework. This corresponds to a non-trivial motion of $\pofw$, since it represents a rotation within the unit cell. %
\end{proof}

 As motivation for the next result, consider an infinitesimally rigid framework $\pofw$ on the $3$-dimensional fixed torus $\Tor^3$ with $|E|=3|V|-3$. The edges of $E$ are therefore independent. By Theorem \ref{thm:nDimConstructive}, every fully-counted subgraph $G' \subseteq G$ satisfying $|E'| = 3|V'| - 3$ has $|\gs(G')| \geq 2$. On the other hand, by Proposition \ref{prop:zeroGains}, any set of edges $E'' \subset E$ with $E'' > 3|V''| - 6$ and $|\gs(E'')| = 0$ is dependent. Therefore, there must be additional conditions on subsets of edges $E'' \subset E$ with $|E''| = 3|V''| - 5$ and $|E''| = 3|V''| - 4$. The following theorem provides necessary conditions on these {\it intermediate} subsets of edges. This is related to the work of Malestein and Theran \cite{Theran} who use a similar ``rank-graded sparsity" idea in their characterization of generic rigidity for $2$-dimensional frameworks on the flexible torus. 
\begin{thm}
Let $\pog$ be a minimally rigid framework on $\Tor^d$. Then for all subsets of edges $Y \subseteq E$, 
\begin{equation}
|Y| \leq d|V(Y)| - {d+1 \choose 2} + \sum_{i=1}^{|\gs(Y)|} (d-i).
\label{eq:dNecessary}
\end{equation}
\label{thm:dNecessary}
\end{thm}

In essence this says that we can add edges beyond what would normally be independent, provided that we also add cycles with non-trivial gains.
Maxwell's condition for finite frameworks in dimension $d$ says that an isostatic framework must satisfy $|E| = d|V| - {d+1 \choose 2}$, and $|E'| \leq d|V'| - {d+1 \choose 2}$ for all induced subgraphs $G' \subseteq G$. Analogously,  a minimally rigid periodic framework in dimension $d$ will have $|E| = d|V|-d$ and induced subgraphs will satisfy $|E'| \leq d|V'| - d$ (Corollary \ref{cor:perMaxwell}).

In addition, we already showed that for a minimally rigid framework $\pog$ on $\Tor^d$:
\begin{enumerate}[(a)]
	\item all induced subgraphs with $|E'| = d|V'| - d$ must have $|\gs(G')| \geq d-1$ (Theorem \ref{thm:nDimConstructive})
	\item any connected subset of edges $Y \subset E$ with $|Y| > d|V(Y)| - {d+1 \choose 2}$ must have $|\gs(Y)| > 0$. (Proposition \ref{prop:zeroGains})
\end{enumerate}

Theorem \ref{thm:dNecessary} extends these results.   We make use of the following simple fact:
\begin{equation}
\textrm{\bf Fact: \ } {d \choose 2} - \sum_{i = 1}^k (d-i) = {d-k \choose 2}.
\label{eqn:Useful}
\end{equation}


\begin{proof}[Proof of Theorem \ref{thm:dNecessary}]
Let $\pog$ be generically minimally rigid on $\Tor^d$, and let $Y \subseteq E$ be a subset of edges. First note that for any subset $Y$ with $|Y| \leq d|V| - {d+1 \choose 2}$, Equation (\ref{eq:dNecessary}) holds trivially. If $|Y| = d|V(Y)|-d$, then the edges of $Y$ are the edges of an induced subgraph, and we must have $|\gs(Y)| \geq d-1$ by (a).

Suppose then that $|Y| = d|V(Y)| - {d+1 \choose 2} + \ell$, where $0 < \ell < {d \choose 2}$. Then for some  $0 < k < d-2$, 
\[\sum_{i=1}^{k} (d-i) \leq \ell < \sum_{i=1}^{k+1}(d-i).\]
Toward a contradiction, suppose that $|\gs(Y)| < k$. We apply the $T$-gain procedure to the edges $Y$, and we obtain gains that are $0$ on more than $d-k$ coordinates. By the arguments of the proof of Theorem \ref{thm:nDimConstructive} for each pair zero coordinates, we can obtain a rotation in that plane. Therefore the space of non-trivial infinitesimal motions of the subset $Y$ on $\Tor^d$ is {\it strictly} larger than ${d-k \choose 2}$. Letting $\mathcal I_k(Y)$ denote the space of non-trivial infinitesimal motions of the subset $Y$, we have shown that $$|\mathcal I_k(Y)| > {d-k \choose 2}.$$

However, since $|Y| < d|V(Y)| - d$, we expect some non-trivial infinitesimal motions of the edges $Y$ on $\Tor^d$.  Since $\pog$ is generically rigid, these motions will disappear when more edges are added to the subset $Y$. How many non-trivial infinitesimal motions would we expect?  An isostatic finite framework with $|E| = d|V| - {d+1 \choose 2}$ has ${d \choose 2}$ non-trivial infinitesimal motions when realized as a periodic orbit framework. Let $\mathcal I(Y)$ denote the space of non-trivial infinitesimal motions we predict based only on the number of edges. Since $\pog$ is minimally rigid, and $|Y| = d|V(Y)| - {d+1 \choose 2} + \ell$, the space of non-trivial infinitesimal motions has dimension $|\mathcal I(Y)| = {d \choose 2} - \ell$.  Now
\begin{eqnarray*}
|\mathcal I(Y)| & = & {d \choose 2} - \ell\\
	 & \leq & {d \choose 2} - \sum_{i=1}^{k} (d-i)\\
 	& = & {d - k \choose 2}\ \ \textrm{by (\ref{eqn:Useful})}\\
	& < & |\mathcal I_k(Y)|.
\end{eqnarray*}
Hence the space of non-trivial infinitesimal motions we expect based on the deficit of edges is smaller than the space predicted by the deficit in the dimension of $\gs(Y)$, which is the contradiction. %
\end{proof}

\subsection{Constructive gain assignments for $d$-periodic orbit frameworks}
\label{sec:nConstructive}

\begin{verse}
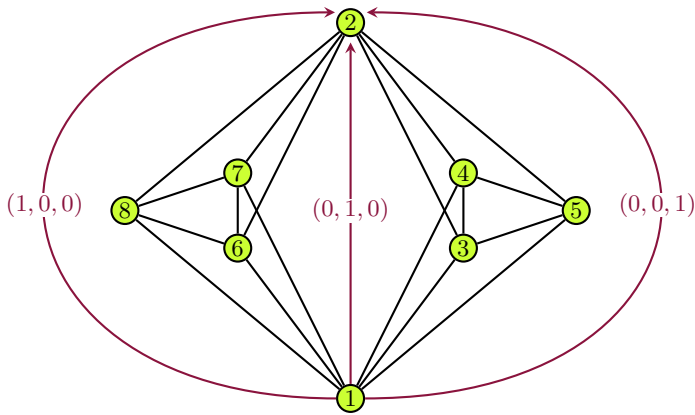
\begin{figure}\begin{center} 
\begin{tikzpicture}[thick, font=\footnotesize]
\tikzstyle{vertex1}=[circle, draw, fill=couch, inner sep=1pt, minimum width=3.5pt]; 
\tikzstyle{vertex2}=[circle, draw, fill=lips, inner sep=1pt, minimum width=3.5pt]; 
\tikzstyle{vertex3}=[circle, draw, fill=melon, inner sep=1pt, minimum width=3.5pt]; 
\tikzstyle{vertex4}=[circle, draw, fill=bluey, inner sep=1pt, minimum width=3.51pt]; 
\tikzstyle{voltage} = [fill=white, inner sep = 0pt,  font=\footnotesize, anchor=center];

\node[vertex1] (1) at (0, 0){$1$}; 
\node[vertex1] (2) at (0, 5) {$2$};
\node[vertex1] (3) at (1.5, 2) {$3$};
\node[vertex1] (4) at (1.5, 3) {$4$};
\node[vertex1] (5) at (3, 2.5) {$5$};
\node[vertex1] (6) at (-1.5, 2) {$6$};
\node[vertex1] (7) at (-1.5, 3) {$7$};
\node[vertex1] (8) at (-3, 2.5) {$8$};

\draw (3) -- (4) -- (5) -- (3);
\draw (2) -- (3) (2) -- (4) (2) -- (5);
\draw (1) -- (3) (1) -- (4) (1) -- (5);

\draw (6) -- (7) -- (8) -- (6);
\draw (2) -- (6) (2) -- (7) (2) -- (8);
\draw (1) -- (6) (1) -- (7) (1) -- (8);

\draw[lips, ->, >=stealth, shorten >=2pt] (1) -- node[voltage] {$(0,1,0)$} (2);
\draw[lips, ->, >=stealth, shorten >=2pt] (1) to [controls=+(180:5.3) and +(180:5.3)] node[voltage] {$(1,0,0)$}  (2);
\draw[lips, ->, >=stealth,shorten >=10pt] (1) to [controls=+(0:5.5) and +(0:5.5)] node[voltage] {$(0,0,1)$}  (2);
\end{tikzpicture}
\caption{An example of a generically flexible periodic orbit graph on $\Tor^3$ with a constructive gain assignment. The black edges form the $3|V| - 6$ ``double bananas" graph, and here we give them gain $(0,0,0)$. The three coloured edges provide the constructive gains. This graph is flexible on $\Tor^3$. \label{fig:doubleBananas}}
\end{center}\end{figure} \end{verse}

We say that $\pog$ has a {\it constructive gain assignment} \index{constructive gain assignment!on $\Tor^d$} if the gain assignment $m$ of $\pog$ is such that (\ref{eq:dNecessary}) is satisfied for every subset $Y$ of edges of $\pog$.  In a subsequent paper, we will demonstrate that a constructive gain assignment on a periodic orbit framework is also sufficient for generic rigidity when $d=2$ (and $d=1$). Unfortunately, the same is not true in higher dimensions. For example, when $d=3$, we can realize the ``double banana" graph as part of a $3|V| - 3$ graph with a constructive gain assignment, as seen in Figure \ref{fig:doubleBananas}. This graph is flexible despite having a constructive gain assignment. The two ``bananas" consisting of all the edges without gains can be rotated independently about the line through vertices 1 and 2. In this way, the sufficiency of any condition for the rigidity of periodic frameworks on $\Tor^d$ depends directly on a characterization of $d$-dimensional finite rigidity. This problem is open for $d >2$.  In the case that $d=1$ or $d=2$, however, combinatorial characterizations of finite rigidity exist, and we will show in a subsequent paper that Theorem \ref{thm:dNecessary} is both necessary and sufficient for infinitesimal rigidity when $d = 1$ or $2$. 

There are, however, gain assignments on the edges of this graph that will produce infinitesimally rigid frameworks on $\Tor^3$. Such frameworks will involve the ``wrapping" of some of the edges of the bananas \index{double bananas!periodic} around the torus. For example, one possible gain assignment is given by the proof of Theorem \ref{waltersThm} (see Section \ref{sec:aKnownResult}), in which the edges of each of the $3$ edge-disjoint spanning trees are assigned the gains $(1,0,0), (0,1,0)$ and $(0,0,1)$ respectively. A similar idea is presented in \cite{BorceaStreinuII}. Notice also that the particular gain assignment produced in the proof of Whiteley's Theorem \ref{waltersThm} is constructive. 

\section{Further work}
\label{sec:furtherWork}
\subsection{Sufficient conditions for rigidity on $\Tor^2$}
We have now seen necessary conditions for the rigidity of frameworks on the fixed torus. As we have noted above, these conditions are clearly not sufficient in general, and moreover, the sufficiency of any condition for the rigidity of periodic frameworks on $\Tor^d$ depends directly on a characterization of $d$-dimensional finite rigidity. This problem is open for $d >2$. In the case that $d=1$ or $d=2$, however, combinatorial characterizations of finite rigidity exist, and we show in a subsequent paper \cite{ThesisPaper2} that Theorem \ref{thm:dNecessary} is both necessary and sufficient for infinitesimal rigidity on the fixed torus when $d = 1$ or $2$. Furthermore, we prove a version of Henneberg's theorem for periodic graphs, demonstrating that all infinitesimally rigid orbit frameworks on $\Tor^2$ can be built up inductively from a sequence of smaller graphs. 


\subsection{Incidentally periodic frameworks}
Throughout this paper, we have been concerned with the topic of {\it forced periodicity}. That is, we have considered periodic frameworks and asked about their rigidity with respect to periodicity-preserving motions. A natural question is about relaxing this restriction to consider {\it incidentally periodic} frameworks, which are infinite frameworks which happen to be periodic, but where we do not require that the periodicity be preserved by infinitesimal motions of the structure.  A question of interest thus becomes: When is a periodic framework flexible, where the flexes may or may not preserve the periodicity of the structure? 

We present a conjecture pertaining to incidentally periodic frameworks:
\begin{conj}
If a framework $\pofw$ is infinitesimally rigid on the flexible torus, then it is infinitesimally rigid as an incidentally periodic (infinite) framework $(\wt G, \wt p)$. 
\end{conj}

\subsection{Periodic bar-body frameworks}
One natural extension of the work in this paper is to {\it periodic bar-body frameworks}.  The generic rigidity of {\it finite} bar-body frameworks is completely characterized in $d$-dimensions with polynomial time algorithms \cite{Tay2}, and a recent proof of the Molecular Conjecture expands this characterization to molecular frameworks \cite{MolecularConj}. \index{Molecular Conjecture} That is, unlike bar-joint frameworks for $d \geq 3$, the generic rigidity of bar-body frameworks for all $d$ can be understood through combinatorial methods alone. 

In particular, Theorem \ref{thm:dNecessary} can be translated to the bar-body setting, where we conjecture that it also provides sufficient conditions for the generic rigidity of periodic bar-body orbit frameworks on $\Tor^d$.

%
%

\bibliographystyle{abbrv} 
\bibliography{../papers}

\end{document}